\documentclass[12pt]{amsart}
\usepackage{geometry}
\usepackage{personalcommands}

\title{A microlocal Cauchy problem through a crossing point of Hamiltonian flows}
\author{Kenta Higuchi}
\address{(K. Higuchi) Faculty of Education, Gifu University, 1-1 Yanagido, Gifu, 501-1193, Japan}
\email{higuchi.kenta.b7@f.gifu-u.ac.jp}

\author{Vincent Louatron}
\address{(V. Louatron) Graduate school of science and engineering,
Ritsumeikan University,
1-1-1 Nojihigashi, Kusatsu, Shiga, 525-8577, Japan}
\email{gr0607hr@ed.ritsumei.ac.jp}

\author{Kouichi Taira}
\address{(K. Taira) Faculty of Mathematics, Kyushu University, 744, Motooka, Nishi-ku, Fukuoka, 819-0395, Japan}
\email{taira.kouichi.800@m.kyushu-u.ac.jp}

\date{}

\begin{document}
    \begin{spacing}{1.2}
        \begin{abstract}
In this paper, we consider $2\times 2$ matrix-valued pseudodifferential equations in which the two characteristic sets intersect with finite contact order. We show that the asymptotic behavior of its solution changes dramatically before and after the crossing point, and provide a precise asymptotic formula. This is a generalization of the previous results for matrix-valued Schr\"odinger operators and Landau-Zener models.
The proof relies on a normal form reduction and a detailed analysis of a simple first-order system.

\end{abstract}
        \maketitle

        \section{Introduction}

\subsection{Main results}\label{SubS:first}

In this paper, we discuss the existence, uniqueness, and asymptotic behavior of solutions of a system of pseudodifferential equations\footnote{The precise definitions of the terms used in this paragraph are provided in Section \ref{S:Preliminaries}.}:
\begin{align}\label{Eq:MicrolocalCauchyProblem}
	\mathscr{P}u\equiv 0 \quad \text{m.l. on}\,\, \Omega
\end{align}
where
\begin{itemize}

\item $\Omega \subset \Rr^2=T^*\Rr$ is a connected open neighborhood of $(0,0)$;

\item $\ms{P}$ is a $2\times 2$-matrix-valued semiclassical pseudo-differential operator of the form 
\begin{equation*}
	\ms{P} =
	\begin{pmatrix}
		P_1 & hQ_1\\
		hQ_2 & P_2
	\end{pmatrix} \quad \text{acting on } L^2(\Rr,\Cc^2),
\end{equation*}
defined as the Weyl quantization (see \eqref{Eq:Weyl}) of the matrix-valued symbol\footnote{Here, we assume that the symbol $p$ is bounded on $\Rr^2$. However, we can still apply our main theorem for unbounded symbols such as $p_j(x,\xi)=\xi^2+V_j(x)$ or $p_j\in S(m)$ for an order function $m$ (see \cite[\S 4.4]{Zwo}) 
since all the statements below are microlocalized on a compact subset of the phase space $T^*\Rr$.}
\begin{equation}\label{Eq:Def-PDO}
p(x,\xi) :=
\begin{pmatrix}
	p_1(x,\xi) & hq_1(x,\xi)\\
	hq_2(x,\xi) & p_2(x,\xi)
\end{pmatrix} \in \mc{M}_2(C^\infty_b).
\end{equation}
\end{itemize}
Throughout the article, we assume that the following two conditions hold.
\begin{condition}[Real-principal type]\label{C:RealPrincipal} 
For each $j=1,2$, the function $p_j$ is of real-principal type at $(0,0)$, that is, 
$p_j$ is real-valued and 
    \begin{equation}
      p_j(0,0)=0,\quad\partial p_j (0,0)\neq0.
    \end{equation}
\end{condition}

\begin{condition}[Finite contact order]\label{C:ContactOrder}
    There exists $m\ge1$ such that $H_{p_1}^kp_2(0,0)=0$ for all $k\in\{0,\ldots,m-1\}$ and
    \begin{equation}\label{Eq:Contact}
        H_{p_1}^{m}p_2(0,0)\neq0.
    \end{equation}
    Here, we write $H_{p_1}:=\partial_{\xi} p_1\partial_{x}-\partial_{x}p_1\partial_{\xi}$. In the following, the \textbf{transversal case} refers to the situation where $m=1$ and the \textbf{tangential case} to situations where $m\geq 2$.
\end{condition}
\noindent
The domain $\Omega$ is chosen to be sufficiently small so that the following conditions hold.
\begin{itemize}
\item The characteristic sets 
\begin{align*}
\Lambda_j:=\{(x,\xi)\in \Rr^2\mid p_j(x,\xi)=0\}\quad (j=1,2)
\end{align*}
are submanifolds and intersect only at $(0,0)$ inside $\Omega$. Geometrically, the intersection is of finite contact order $m$ by Condition \ref{C:ContactOrder}. We refer to $m$ as \textbf{the contact order} of $\Lambda_1$ and $\Lambda_2$;

\item For $j=1,2$, the portion $(\Omega\cap\Lambda_j)\setminus \{(0,0)\}$ is divided into two connected components $\gamma_j^{\inc}$ and $\gamma_j^{\out}$, where the Hamiltonian flow generated by $H_{p_j}$ is incoming to the crossing point $(0,0)$ on $\gamma_j^\inc$ while it is outgoing from $(0,0)$ on $\gamma_j^\out$. See Figure~\ref{Fig:notations}.
\end{itemize}

\begin{figure}[h]
    \begin{minipage}{\linewidth}
    {\centering
	\includegraphics[width=\linewidth,page=2]{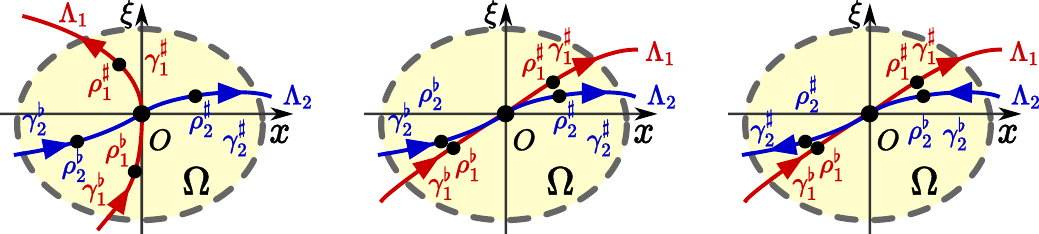}
    \caption{Hamiltonian flows and notations}\label{Fig:notations}
    }
    \end{minipage}
\end{figure}

In the setting above, the uniqueness, existence, and asymptotic behavior of microlocal solutions \textbf{away from the crossing point $(0,0)$} are well known and can be analyzed using the usual WKB method. In particular, it is known that microlocal solutions of \eqref{Eq:MicrolocalCauchyProblem} are microlocalized on $\Lambda_1 \cup \Lambda_2$ --- meaning that their wavefront set is included in the characteristic sets, see Definition \ref{Def:MicrolocallySmall}. 
On each connected component $\gamma_j^\dir$ $(j=1,2$, $\dir\in\{\inc,\out\})$, one has:
\begin{itemize}
\item (Uniqueness) The space of microlocal solutions near 
$\gamma_j^\dir$ 
is \textbf{one-dimensional}. 
More precisely, $u\equiv0$ microlocally near a point of $\gamma_j^\dir$ implies that $u\equiv0$ microlocally on all $\gamma_j^\dir$.

\item (Existence) 
There exists a non-trivial solution $w_j^{\dir}$ of \eqref{Eq:MicrolocalCauchyProblem} microlocally on $\gamma_j^{\dir}$, 
uniquely determined up to a multiplicative constant.

\item (Asymptotic behavior) The solution $w_j^{\dir}$ is a WKB state (or Lagrangian distribution) associated with the Lagrangian submanifold $\Lambda_j$.
Generically for $p_1$ and $p_2$, we can write $w_j^{\dir}(x)=\sigma_j^{\dir}(x)e^{\frac{i}{h}\phi_j(x)}$, where $\phi_j$ is a solution to the eikonal equation $p_j(x,\phi_j'(x))=0$ and the amplitude $\sigma_j^{\dir}$ satisfies a transport equation. 
\end{itemize}
According to the uniqueness and existence written above, for each solution $u$ of \eqref{Eq:MicrolocalCauchyProblem} microlocally on $\Omega$, there exist four complex numbers $\alpha_1^\inc$, $\alpha_2^\inc$, $\alpha_1^\out$, $\alpha_2^\out\in\mathbb{C}$  such that 
\begin{align}\label{Eq:Defalpha}
    u\equiv \alpha_j^{\dir}w_j^{\dir}
    \quad\text{ m.l. on } \gamma_j^{\dir}\quad (j=1,2,\ \dir\in\{\inc,\out\}).
\end{align} 
See Proposition \ref{Prop:matrixmicsol} for the more rigorous statements, proved by standard WKB analysis.

The purpose of this paper is to study the equation \eqref{Eq:MicrolocalCauchyProblem} \textbf{near the crossing point $(0,0)$}.
More precisely, we consider the following problems.

\underline{\textbf{Problem 1}}: How many microlocal solutions are there for \eqref{Eq:MicrolocalCauchyProblem}?

\underline{\textbf{Problem 2}}: 
What is the relationship between the coefficients $\alpha_j^{\dir}$?

\noindent
Problem $1$ is solved by the following theorem.

\begin{thm}[Existence and uniqueness of the microlocal Cauchy problem]\label{Thm:ExistenceAndUniqueness}
In the above setting,
the following holds: 

\noindent$(i)$ $($Uniquness$)$
    Let $u$ be a microlocal solution of $\ms{P}u=0$ on $\Omega$. 
    If there exist two points $\rho_1^\inc\in\gamma_1^\inc$ and $\rho_2^\inc\in\gamma_2^\inc$ such that $u\equiv 0$ microlocally at $\rho_1^\inc$ and $\rho_2^\inc$, then 
    \begin{equation*}
        u\equiv0\quad\text{m.l. on }\Omega.
    \end{equation*}

\noindent$(ii)$ (Existence) Fix two WKB states $w_j^{\inc}$ ($j=1,2$), 
which are tempered and satisfy \eqref{Eq:MicrolocalCauchyProblem}  
microlocally
on $\gamma_j^{\inc}$. 
Then, for each $\alpha_1^{\inc},\alpha_2^{\inc}\in\mathbb{C}$, there exists a tempered $u$ such that
\begin{align}\label{eq:msPCauchy}
\ms{P}u\equiv 0\quad \text{m.l. on }\Omega,\quad u\equiv \alpha_{j}^{\inc}w_j^{\inc}\quad \text{m.l. on }\gamma_j^{\inc}\quad (j=1,2).
\end{align}
We refer to the constants $\alpha_1^\inc$ and $\alpha_2^\inc$ appearing in the second identity \eqref{eq:msPCauchy} as the microlocal Cauchy data of the system \eqref{Eq:MicrolocalCauchyProblem}.

\end{thm}

In particular, the space of microlocal solutions on $\Omega$ of \eqref{Eq:MicrolocalCauchyProblem} is \textbf{two-dimensional}.
The term microlocal Cauchy data is taken from \cite{BFRZ}, where a microlocal Cauchy problem involving a scalar Hamiltonian with a hyperbolic fixed point is studied.


We now turn to Problem $2$. Fix two WKB states $w_j^{\dir}$ ($j=1,2$), which are tempered and satisfy \eqref{Eq:MicrolocalCauchyProblem} microlocally on $\gamma_j^{\dir}$.
Thanks to Theorem \ref{Thm:ExistenceAndUniqueness} (\textit{ii}), for given $\alpha_1^{\inc},\alpha_2^{\inc}\in\mathbb{C}$, we can find a solution $u$ of \eqref{eq:msPCauchy}. As stated above in \eqref{Eq:Defalpha},
there exist $\alpha_1^{\out},\alpha_2^{\out}\in\mathbb{C}$ such that
\begin{align}\label{eq:outgoing-data}
u\equiv \alpha_{j}^{\out}w_j^{\out}\quad \text{m.l. on }\gamma_j^{\out}\quad (j=1,2).
\end{align}
By virtue of Theorem \ref{Thm:ExistenceAndUniqueness} $(i)$, the constants $\alpha_1^{\out},\alpha_2^{\out}\in\mathbb{C}$ are unique modulo $\BigO{h^\infty}$. Since the equation \eqref{eq:msPCauchy} is linear, we can find a $2\times 2$ matrix $T$, called \textbf{the transfer matrix} such that
\begin{equation*}
    \begin{pmatrix}
        \alpha_1^\out\\
        \alpha_2^\out
    \end{pmatrix}
    =T
    \begin{pmatrix}
        \alpha_1^\inc\\
        \alpha_2^\inc
    \end{pmatrix},
\end{equation*}
which can be defined modulo $\BigO{h^{\infty}}$. Our next theorem describes the asymptotic behavior of the transfer matrix. Here, we remark that the transfer matrix $T=T(w_1^\inc,w_2^\inc,w_1^\out,w_2^\out)$ depends on the choice of the WKB states $w_j^{\dir}$ ($j=1,2$, $\dir\in \{\inc,\out\}$), as described in \eqref{Eq:TransferMatrixChangeWKB} in the example below. To simplify the connection formula, we will choose these WKB states as in the definition below.

\begin{definition}\label{def:normalized}
\noindent$(i)$ Let $v_j$ be a microlocal solution to the scalar equation $Pv_j\equiv 0$ m.l. on $\Omega$. We say that $v_j$ is \textbf{normalized} if
\begin{equation}\label{eq:FBI-normal}
    \frac {e^{-i\Theta(p_j)/2}}{\sqrt{2\pi h}}\int_{\Rr}e^{-\frac{x^2}{2h}}v_j(x)dx=1+\BigO{h},
\end{equation}
where we put
\begin{equation}\label{eq:def-Theta}
    \Theta(p_j):=\left\{
    \begin{aligned}
    &-\arctan\left(\frac{\partial_x p_j(0,0)}{\partial_\xi p_j(0,0)}\right)
    &&(\partial_\xi p_j(0,0)\neq0),\\
    &-\frac{\pi}2&&(\partial_\xi p_j(0,0)=0).
    \end{aligned}\right.
\end{equation}

\noindent$(ii)$ Let $w_j^{\dir}$ ($j=1,2$, $\dir\in \{\inc,\out\}$) be a microlocal solution to the equation \eqref{Eq:MicrolocalCauchyProblem} microlocally on $\gamma_j^{\dir}$. The solution $w_{j}^{\dir}$ is said to be \textbf{normalized} if
\begin{align}\label{eq:WKB-normal}
w_j^\dir\equiv v_j\mathbf{e}_j+\BigOO{h}{L^2(\Rr;\Cc^2)}\quad \text{m.l. on } \gamma_j^{\dir},
\end{align}
where the pair of $\mathbf{e}_1={}^t(1,0)$ and $\mathbf{e}_2={}^t(0,1)$ is the canonical basis of $\Cc^2$, and $v_j$ is a normalized microlocal solution to the scalar equation $P_jv_j\equiv0$ m.l. on $\Omega$.
\end{definition}

\begin{rmk}
Normalized solutions indeed exist (Propositions~\ref{prop:scalarmicsol} and \ref{Prop:matrixmicsol}) and are unique up to $\BigO{h}$. 
The ambiguity involving an $\BigO{h}$-term comes from that in \eqref{eq:FBI-normal} and \eqref{eq:WKB-normal}.

\end{rmk}

We now state our second result which deals with Problem $2$. The symbol $\delta_{m,1}$ stands for Kronecker's delta: $\delta_{m,1}=1$ for $m=1$ and $\delta_{m,1}=0$ for $m\geq 2$.

\begin{thm}[Microlocal connection formula]\label{Th:LSM}
Suppose that the WKB solutions $w_j^{\dir}$ are normalized in the sense of Definition \ref{def:normalized}.
Then, the transfer matrix $T=T(w_1^\inc,w_2^\inc,w_1^\out,w_2^\out)$ satisfies the asymptotic formula
    \begin{equation}\label{eq:asymptotics-T}
        T=\begin{pmatrix}
            1&0\\0&1
        \end{pmatrix}
        -ih^{\frac1{m+1}}\begin{pmatrix}
            0&\omega_1 q_1(0,0)\\ \omega_2q_2(0,0)&0
        \end{pmatrix}
        +\BigO{h^{\frac2{m+1}}\left(\log\frac1h\right)^{\delta_{m,1}}},
    \end{equation}
    where the $q_j$ are the anti-diagonal terms of $p$ defined in \eqref{Eq:Def-PDO} and the $\omega_j$ are given by
    \begin{align*}
        &\omega_1=2\mu_m\left(\frac{-\operatorname{sgn}(sH_{p_1}^m p_2(0,0))\pi}{2(m+1)}\right)\bm{\Gamma}\left(\frac{m+2}{m+1}\right)
        \left(\frac{|\partial_{(x,\xi)}p_2(0,0)|}{|\partial_{(x,\xi)}p_1(0,0)|}\frac{(m+1)!}{|H_{p_1}^mp_2(0,0)|}\right)^{\frac1{m+1}},\\
        &\omega_2=2\mu_m\left(\frac{-\operatorname{sgn}(sH_{p_2}^m p_1(0,0))\pi}{2(m+1)}\right)\bm{\Gamma}\left(\frac{m+2}{m+1}\right)
        \left(\frac{|\partial_{(x,\xi)}p_1(0,0)|}{|\partial_{(x,\xi)}p_2(0,0)|}\frac{(m+1)!}{|H_{p_2}^mp_1(0,0)|}\right)^{\frac1{m+1}}.
    \end{align*}
Here, the constant $s$ and the function $\mu_m=\mu_m(\theta)$ are defined by
    \begin{equation}\label{eq:def-s-tau-mu}
        s:=\operatorname{sgn}\prod_{j=1}^2\left(H_{p_j}(0,0),\tau(\Theta(p_j))\right)_{\Rr^2}
        ,\quad\tau(\theta)=(\cos\theta,\sin\theta),\qquad
        \mu_m(\theta):=\frac{e^{i\theta}+e^{i(-1)^{m+1}\theta}}2.
    \end{equation}
where $(\cdot,\cdot)_{\Rr^2}$ stands for the usual (Euclidean) inner product in $\Rr^2$.
\end{thm}

\begin{rmk}
The ambiguity of the choices of normalized solutions $w_j^{\dir}$ is not explicit in this formula since the estimate of the remainder $\BigO{h^{2/(m+1)}\left(\log h^{-1}\right)^{\delta_{m,1}}}$ is rougher than $\BigO{h}$.
\end{rmk}

\begin{rmk}
    By the definition \eqref{eq:def-Theta} of $\Theta(p_j)$, the vector $\tau(\theta)$ defined by \eqref{eq:def-s-tau-mu} is the unit vector parallel to the Hamiltonian vector $H_{p_j}(0,0)$, that is,
    \begin{equation*}
        H_{p_j}(0,0)=\pm\left|H_{p_j}(0,0)\right|\tau(\Theta(p_j)),\quad
        \left(H_{p_j}(0,0),\tau(\Theta(p_j))\right)_{\Rr^2}=\pm\left|H_{p_j}(0,0)\right|,
    \end{equation*}
    where the identities hold true for the same sign.
    The sign is positive if and only if either $\partial_\xi p_j(0,0)>0$ or $\partial_\xi p_j(0,0)=0$ and $\partial_xp_j(0,0)>0$ holds (see Figures~\ref{Fig:Sign-s-transv} and \ref{Fig:Sign-s-tang}).
    In the tangential case, one has $\Theta(p_1)=\Theta(p_2)$. 
    This implies the identity
    \begin{equation*}
        s=\operatorname{sgn}(H_{p_1}(0,0),H_{p_2}(0,0))_{\Rr^2}.
    \end{equation*}
\end{rmk}

\begin{rmk}
    The formulae of $\omega_1$ and $\omega_2$ in Theorem~\ref{Th:LSM} are written symmetrically.
    We can also write them using only either one of $H_{p_1}^mp_2(0,0)$ or $H_{p_2}^mp_1(0,0)$ 
    since one has the identity
    \begin{equation*}
        H_{p_1}^mp_2(0,0)=-\left(\frac{|\partial_{(x,\xi)}p_1(0,0)|}{|\partial_{(x,\xi)}p_2(0,0)|}\right)^{m-1}H_{p_2}^mp_1(0,0).
    \end{equation*}
    This identity will be proven in Lemma~\ref{lem:prelim-for-compute}.
\end{rmk}

\begin{figure}[h]
    \begin{minipage}{\linewidth}
    {\centering
	\includegraphics[width=\linewidth,page=2]{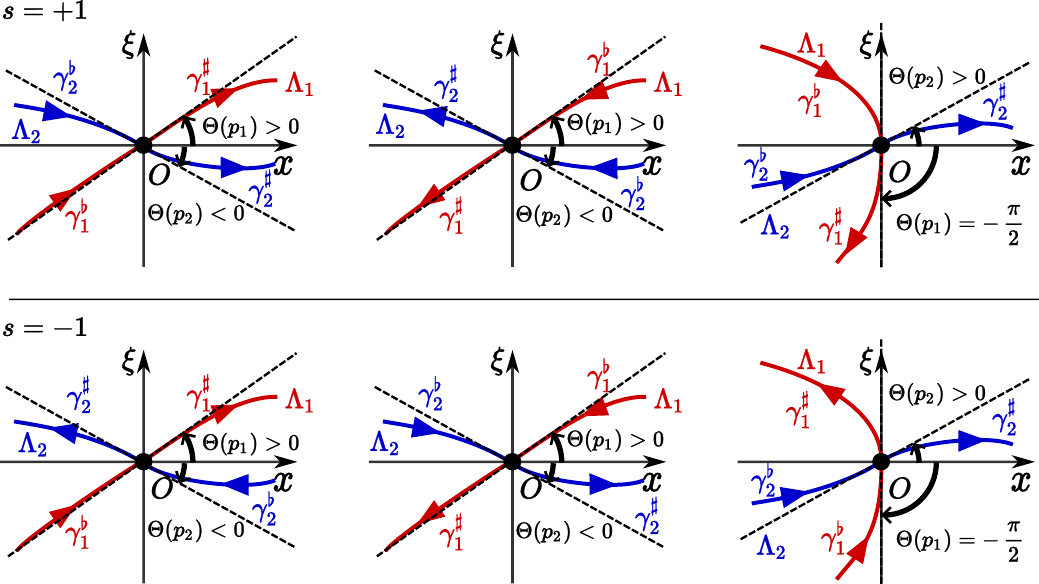}
    \caption{Hamiltonian flows, the sign of $s$ introduced by \eqref{eq:def-s-tau-mu}, and $\Theta(p_j)$ introduced by \eqref{eq:def-Theta} for the transversal case $m=1$.}\label{Fig:Sign-s-transv}
    }
    \end{minipage}
\end{figure}

\begin{figure}[h]
    \begin{minipage}{\linewidth}
    {\centering
    \includegraphics[width=\linewidth,page=2]{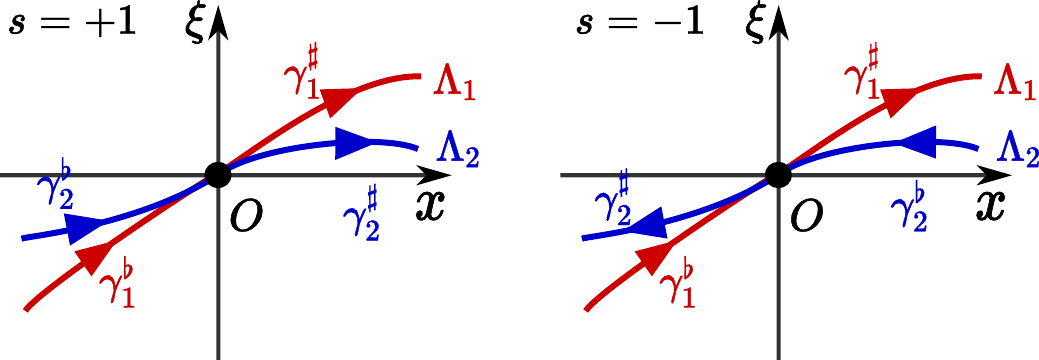}
    \caption{Hamiltonian flows and the sign of $s$ introduced by \eqref{eq:def-s-tau-mu} for the tangential case $m\ge2$.}\label{Fig:Sign-s-tang}
    }
    \end{minipage}
\end{figure}

\subsection{Motivations and context}

The main motivation to computing the transfer matrix \eqref{eq:asymptotics-T} is to provide a general method for computing spectral asymptotics, namely the eigenvalue splitting (\cite{AF}) and the distribution of resonances (\cite{FMW1, FMW2, FMW3, AFH, AFH2, Lou}) of matrix-valued Schr\"odinger operators in the semiclassical limit $h \to +0$\footnote{We also find a similar approach for the asymptotics of the scattering matrix of matrix-valued Schr\"odinger operators 
with
transversal crossing of Hamiltonian flows in \cite{ABA}.}. 
Essentially, such distribution is obtained by studying the associated eigenstates and resonant states, which are microlocalized on the characteristics $\Lambda_j$ and behave as microlocal solutions of the eigenvalue equation \eqref{Eq:MicrolocalCauchyProblem}. The behavior of such solutions on the characteristics is known outside crossing points, therefore establishing a microlocal connection formula at those crossing points is enough to describe resonant and eigen- states. 
Such an approach is employed in \cite{FMW3,AF,AFH,AFH2,Lou}. While \cite{FMW3} and \cite{AF} used a normal form reduction to obtain a microlocal connection formula for transversal case, \cite{AFH,AFH2,Lou} used a non-microlocal method to obtain the formula for tangential case.
Our goal is to unify those different methods using a normal form reduction which applies to the both situations.
The study of a microlocal problem via a normal form reduction is not new, and the idea traces back to the works of Landau and Zener on avoided crossings. We distinguish two types of normal form reductions of matrix-valued operators:
\begin{itemize}
\item The first essentially consists in reducing a matrix-valued symbol of principal type to a scalar-valued symbol of the form $x\xi$, which has an hyperbolic fixed point at $(0,0)$. Applications of this method to microlocal problems are found in \cite{FMW3} or \cite{CLP} for example, where the interaction term is assumed to be elliptic at the crossing point. 

\item The second method consists in reducing the matrix-valued symbol into another simpler matrix-valued symbol. This method has the advantage of preserving the geometric nature of the crossing point (the crossing of real principal type is not transformed into an hyperbolic fixed point) and does not require ellipticity. Such method was developed for codimension 3 crossings in \cite{CdV03}.
This was applied in \cite{CdV11} for 1-dimensional, transversal crossings. 
\end{itemize}
For both methods however, tangential crossings were excluded up until now. 
To deal with such crossings, the first author extended in \cite{AFH} the construction of \cite{FMW1} of exact solutions near the crossing point. 
The second author improved this method in \cite{Lou} 
under a relaxed elliptic assumption on the interaction term.

Our normal form result, Theorem \ref{Thm:Matrix-NF}, accounts for both transversal and tangential crossings and holds without any ellipticity assumption. It relies essentially on the Darboux theorem and the Egorov theorem, and makes for a much simpler proof than that of \cite{CdV11} for the transversal case. 
Moreover, we apply this normal form reduction to recover the results of \cite{AFH}, \cite{FMW1}, \cite{FMW2}, and \cite{FMW3} all at once, through the microlocal connection formula (Theorem \ref{Th:LSM}). In practice, since $P$ is a $2\times2$ matrix-valued differential operator of order $1$ at the principal level, we end up computing $2 \times 2$ change-of-basis matrices (see \eqref{lem:basisexist}), whereas for example the direct computation of \cite[Proposition 3.12]{AFH} for the Schr\"odinger operator (which is of order $2$ at the principal level) involves $4 \times 4$ change-of-basis matrices. 

In the rest of the section, we sketch an outline of the connection formula through a simple example and we apply Theorem \ref{Th:LSM} to recover the connection formulae for Schr\"odinger-type operators. In Section \ref{S:Preliminaries}, we lay out the background of microlocal analysis needed for this paper. The organization of sections \ref{S:NormalForms}, \ref{S:CFReduced} and \ref{S:Correspondence} will be detailed in Section \ref{SubS:OutlineProof}.

\subsection{Simple case (non-microlocal)}\label{SubS:simple}

Let us describe our method for a simple case with
\begin{align*}
P_1=hD_x,\quad P_2=hD_x-f(x),\quad R_1=R_2=r(x)
\end{align*}
where $f\in C^{\infty}(\Rr)$ and $r\in C_c^{\infty}(\Rr)$ are real-valued functions satisfying
\begin{align*}
f^{(k)}(0)=0\quad \text{for}\quad k=0,1,\hdots,m-1,\quad f^{(m)}(0)\neq 0,\qquad f(x)\neq 0\quad\text{for}\quad x\in \operatorname{supp}r\setminus\{0\}.
\end{align*}
Although Theorems \ref{Thm:ExistenceAndUniqueness} and \ref{Th:LSM} deal with microlocal solutions, we focus on the exact solutions here.

The equation $\mathscr{P}u=0$ with $u(x)={}^t(u_1(x),u_2(x))$ is written as the following system of differential equations:
\begin{align}\label{eq:intromodeleq}
\left\{
\begin{aligned}
    &hD_xu_1(x)=-hr(x)u_2(x)\\
    &(hD_x-f(x))u_2(x)=-hr(x)u_1(x).
\end{aligned}\right.
\end{align}
First, we observe that this equation can be written by the Duhamel formula as
\begin{align}\label{eq:intromodelDuh}
u_1(x)=\alpha^\inc_1-i\Gamma_+(e^{-\frac{i}{h}F(\cdot)}u_2),\quad u_2(x)=e^{\frac{i}{h}F(x)} \left(\alpha^\inc_2-i\Gamma_-(u_1)\right),
\end{align}
with $\alpha^\inc_1=u_1(-1)$ and $\alpha^\inc_2=u_2(-1)\exp(\frac ih\int_{-1}^0f(y)dy)$, where we define
\begin{equation*}
    \Gamma_\pm v(x):=\int_{-1}^x e^{\pm F(y)}r(y)v(y)dy,\quad F(x):=\int_0^xf(y)dy.
\end{equation*}

\subsection*{Existence and uniqueness of exact solutions}
The equation \eqref{eq:intromodeleq} is a first order differential equation and the space of the solutions is two-dimensional. Moreover, the solution is uniquely determined by the initial data $\alpha_1^{\inc},\alpha_2^{\inc}\in\mathbb{C}$ and can be written explicitly as
\begin{align}\label{eq:intromodelDuh2}
u_1=(I+\Gamma_+\Gamma_-)^{-1}\left(\alpha^\inc_1-i\alpha^\inc_2\Gamma_+(\mathbbm{1})\right),\,\, u_2=e^{\frac{i}{h}F}(I+\Gamma_-\Gamma_+)^{-1}\left(\alpha^\inc_2-i\alpha^\inc_1\Gamma_-(\mathbbm{1})\right)
\end{align}
due to \eqref{eq:intromodelDuh}, where $\mathbbm{1}$ is the constant function valued at $1$. Here, the operator $I+\Gamma_+\Gamma_-$ is invertible for sufficiently small $h>0$, in fact, we have $\|\Gamma_\pm\Gamma_\mp\|_{L^\infty\to L^\infty}=\BigO{h^{\frac1{m+1}}}$
by the degenerate stationary phase formula (see Proposition \ref{prop:normbdd}).

\subsection*{Asymptotic behavior of the solution}
To describe the asymptotic behavior of the solution, we need more precise estimates for 
$\Gamma_{\pm}$. We have
$\|\Gamma_\pm\Gamma_\mp(\mathbbm{1})\|_{L^\infty\to L^\infty}=\BigO{h^{\frac2{m+1}}\logdelt }$ and
\begin{align*}
\Gamma_\pm(\mathbbm{1})(x)=\left\{
    \begin{aligned}
        &
        \BigO{h}
        &&(-1\le x\le -\epsilon)\\
        &
        h^{\frac1{m+1}}r(0)\omega_{\pm}+\BigO{h^{\frac{2}{m+1}}\logdelt}
        &&(\epsilon\le x\le 1)
    \end{aligned}\right.\quad \text{for fixed $0<\epsilon<1$}
\end{align*}
with $\omega_{+}=\omega, \omega_-=\overline{\omega}$ and $\omega$ defined by 
\eqref{Eq:omega}. This is a consequence
of the degenerate stationary phase formula (see Proposition \ref{prop:normbdd}), where we note that the critical point of $F$ on $\supp r$ is $x=0$ only.
 By \eqref{eq:intromodelDuh2} and the Neumann series expansion $(I+\Gamma_{\pm}\Gamma_{\mp})^{-1}=\sum_{k=0}^{\infty}(-1)^k(\Gamma_{\pm}\Gamma_{\mp})^k$, we conclude
\begin{align}\label{eq:introasyexp}
\begin{pmatrix}u_1(x)\\ e^{-\frac{i}{h}F(x)}u_2(x)\end{pmatrix}=&\left\{\begin{aligned}
        &
\begin{pmatrix}\alpha^\inc_1\\ \alpha^\inc_2\end{pmatrix}
&&(-1\le x\le -\epsilon)\\
&\begin{pmatrix}\alpha^\inc_1-ih^{\frac1{m+1}}r(0)\omega \alpha^\inc_2\\ 
\alpha^\inc_2-ih^{\frac{1}{m+1}}r(0)\overline{\omega} \alpha_1^{\inc}\end{pmatrix}
&&(\epsilon\le x\le 1)
\end{aligned}\right..
\end{align}
modulo $\BigO{h^{\frac{2}{m+1}}\logdelt }$.
In particular, we see that the asymptotic behavior of the solution $u$ dramatically changes around $x=0$.

\subsection*{Transfer matrix}
Now we describe the asymptotic behavior of the transfer matrix $T$ with respect to the basis $\tw_1^{\inc},\tw_2^{\inc}$ and $\tw_1^{\out},\tw_2^{\out}$ that are WKB solutions for $-1\leq x\leq -\epsilon$ and $\epsilon\leq x\leq 1$ respectively of the form $\tw_1^{\dir}={}^t(1,0)+\BigO{h}$ and $\tw_2^{\dir}={}^t(0,e^{\frac{i}{h}F(x)})+\BigO{h}$ for $\dir\in\{\inc,\out\}$.
Note here that we are using the functions $\tw_2^\inc$, $\tw_2^\out$ not normalized in the sense of Definition~\ref{def:normalized} for a notational simplicity (the normalization requires a multiplication by a constant depending on $f'(0)$).
The asymptotic expansion \eqref{eq:introasyexp} yields
\begin{align*}
u(x)=\left\{\begin{aligned}
&\alpha_1^{\inc}\tw_1^{\inc}+\alpha_2^{\inc}\tw_2^{\inc} &&(-1\le x\le -\epsilon)\\
&(\alpha^\inc_1-ih^{\frac1{m+1}}r(0)\omega \alpha^\inc_2)\tw_1^{\out}+(\alpha^\inc_2-ih^{\frac{1}{m+1}}r(0)\overline{\omega} \alpha_1^{\inc})\tw_2^{\out} &&(\epsilon\le x\le 1)
    \end{aligned}\right.
\end{align*}
modulo $\BigO{h^{\frac{2}{m+1}}\logdelt }$. Then, the transfer matrix $T$ can be calculated as
\begin{align*}
\begin{pmatrix}
\alpha_1^{\inc}\\
\alpha_2^{\inc}
\end{pmatrix}\mapsto \begin{pmatrix}
\alpha^\inc_1-ih^{\frac1{m+1}}r(0)\omega \alpha^\inc_2\\
\alpha^\inc_2-ih^{\frac{1}{m+1}}r(0)\overline{\omega} \alpha_1^{\inc}
\end{pmatrix}=\left(I_2-ih^{\frac{1}{m+1}}\begin{pmatrix}0&r(0)\omega  \\ r(0)\overline{\omega}&0\end{pmatrix} \right)\begin{pmatrix}
\alpha_1^{\inc}\\
\alpha_2^{\inc}
\end{pmatrix}
\end{align*}
modulo $\BigO{h^{\frac{2}{m+1}}\logdelt }$.

\subsection{The example of the matrix-valued Schr\"odinger operator}\label{SubS:Schrodinger}

We now turn to a microlocal example. We consider a matrix-valued Schr\"odinger operator
\begin{equation}\label{Eq:MatrixSchrödingerOperator} P_{\rm Sc} = P_{\rm Sc}(h,E_0) := 
	\begin{pmatrix}
		P_1-E_0 & h W\\
		h W & P_2-E_0
	\end{pmatrix}
\end{equation}
with $P_j :=(h D_x)^2  + V_j(x)$ $j = 1,2$, $E_0 \in \Rr$ and $W = W(x)$ is a multiplication operator. Below, we apply Theorem \ref{Th:LSM} and recover the transfer matrix of $P_{\rm Sc}$ at the crossing points as obtained in \cite[Theorem 3.4]{AFH} and \cite[Theorem 2]{AFH2}.
 
Assume that the function $V_2-V_1$ vanishes only at $x=0$, at a finite order $n \in \Nn \setminus \{0\}$:
\begin{equation}\label{Eq:ContactOrderOfPotentials}
	(V_2-V_1)^{(j)}(0) = 0 \text{ for all }  j \in \{0, \dots, n-1\}, \stext{and} (V_2-V_1)^{(n)}(0) \neq 0.
\end{equation}
For simplicity, assume that $V_1(0) = 0$ and $|x^{-1}V'_j(x)| > 0$ near $0$.
The characteristic sets $\Lambda_j := \{\xi^2 + V_j(x) = E_0\}$ associated with $P_j-E_0$, $j = 1,2$, are assumed to satisfy:
\begin{enumerate}[label=\textbf{\blue{(\roman*)}}]
	\item If $E_0 > 0$ then $\Lambda_1 \cap \Lambda_2 = \{(0,\xi_0), (0,-\xi_0)\}$. This is the situation considered in \cite{AFH}.\label{item:NoCaustic}
	\item If $E_0 = 0$ then $\Lambda_1 \cap \Lambda_2 = \{(0,0)\}$. This is the situation considered in \cite{AFH2}.\label{item:Caustic}
\end{enumerate}
where $\xi_0 = \sqrt{E_0}$. The crossing order $m$ of $\Lambda_1$ and $\Lambda_2$ is
$m = n$ for (i) and $m=2n$ for (ii). In (ii), the crossing point $(0,0)$ is also a turning point (or caustic), meaning that $\partial_\xi p_j(0,0)=0$.

\subsection*{Transfer matrix for \ref{item:NoCaustic}}

The computation of the transfer matrix for both crossing points is similar (see Remark \ref{Rem:ConjugateTransferMatrix}), therefore we only detail computations for the transfer matrix at the crossing point $(0,\xi_0)$. Since $(0,\xi_0)$ is not a turning point, we may solve the eikonal and transport equations and fix a basis of WKB solutions $w_1^{\inc}$, $w_2^{\inc}$,  $w_2^{\out}$ and $w_2^{\out}$ for $P_{\rm Sc}$ of the form
\begin{equation}\label{Eq:WKBSchrodinger}
    w_j^\dir(x,h) = e^{\frac{i}{h}\int_0^x\sqrt{E_0-V_j(y)}\,dy}
    \sigma_j(x)(\mathbf{e}_j + \BigO{h})
\end{equation}
where
\begin{equation}\label{Eq:PhaseAmpWKBSchrod}
    \sigma_j(x) := c_j\left(1-\frac{V_j(x)}{E_0}\right)^{-1/4} \quad \text{and }
    c_j := \sqrt{\frac{\left|\partial_{(x,\xi)} p_j(0,\xi_0)\right|}{\left|\partial_\xi p_j(0,\xi_0)\right|}} = \left(1 + \frac{V_j'(0)^2}{4E_0}\right)^{\frac 14}.
\end{equation}
According to Proposition \ref{prop:scalarmicsol}, the identity $\sigma_j(0) = c_j$ ensures that those solutions are normalized in the sense of Definition \ref{def:normalized}. Now, fix a neighborhood $\Omega$ of $(0,\xi_0)$. Theorem \ref{Thm:ExistenceAndUniqueness} ensures there exist constants $\alpha_j^\dir = \alpha_j^\dir(u) \in \Cc$ such that, for any microlocal solution $u$ of $P_{\rm Sc}u \equiv 0$ near $\Omega$,
\begin{equation*}
	u \equiv \alpha_j^\dir w_j^\dir \text{ microlocally near } \rho_j^\dir.
\end{equation*}
The transfer matrix $T_{\rm Sc}(0,\xi_0)$ of $P_{\rm Sc}$ at $(0,\xi_0)$ is uniquely defined by ${}^t(\alpha_1^\out, \alpha_2^\out) = T_{\rm Sc}(0,\xi_0) {}^t(\alpha_1^\inc, \alpha_2^\inc)$. We recall that the crossing order is $m = n$ and that
\begin{equation*}
    \partial_{x,\xi} p_j(0,\xi_0) = (V_j'(0),2\xi_0) \quad \text{and } H_{p_1}^np_2(0,\xi_0) = 2^n\xi_0^n(V_2-V_1)^{(n)}(0) \quad \text{and } s = 1.
\end{equation*}
We then apply Theorem \ref{Th:LSM} and recover the result of \cite{AFH} on the asymptotics of $T_{\rm Sc}(0,\xi_0)$: 
\begin{prop}
	In the case \ref{item:NoCaustic}, the transfer matrix $T_{\rm Sc}(0,\xi_0)$ satisfies
	\begin{equation}\label{Eq:TransferMatrixSchrodinger}
		T_{\rm Sc}(0,\xi_0)=I_2-ih^{\frac1{n+1}}
		\begin{pmatrix}
			0&\omega W(0)\\\overline{\omega}W(0)&0
		\end{pmatrix}
		+\BigOO{h^{\frac2{n+1}}\left(\log\frac1h\right)^{\delta_{n,1}}}{\Cc^2\to\Cc^2},
	\end{equation}
	as $h\to+0$ with 
	\begin{align}\label{Eq:OmegaSchrodinger+}
		\omega=\mu_n\left(\frac{-\sgn((V_2 - V_1)^{(n)}(0))}{2(n+1)}\pi\right)
		\mathbf{\Gamma}\left(\frac{n+2}{n+1}\right) 
		\left(\frac{2(n+1)!}{|(V_2 - V_1)^{(n)}(0)|}\right)^{\frac{1}{n+1}}. 
	\end{align}
\end{prop}


\begin{rmk}\label{Rem:TranslateT}
	Rigorously speaking, one should first translate the crossing point to $(0,0)$ by conjugating $P_{\rm Sc}$ by the Fourier Integral Operator $e^{\frac{i}{h}x\xi_0}$ generated by the symplectomorphism $\kappa(x,\xi) := (x,\xi-\xi_0)$. Then, Theorem \ref{Th:LSM} applied to $\ms{P} := e^{-\frac{i}{h}x\xi_0} \, P_{\rm Sc} \, e^{\frac{i}{h}x\xi_0}$ yields the transfer matrix of $\ms{P}$. Deducing the transfer matrix $P_{\rm Sc}$ from that of $\ms{P}$ is straightforward (see Proposition \ref{Prop:FIO-WKB}). 
\end{rmk}

\begin{rmk}\label{Rem:ConjugateTransferMatrix}
Note that $H_{p_1}^np_2(0,-\xi_0) = (-1)^nH_{p_1}^np_2(0,\xi_0)$ so that
	\begin{equation*}
		T_{\rm Sc}(0,-\xi_0) = I_2-ih^{\frac1{n+1}}
		\begin{pmatrix}
			0&\overline{\omega}W(0)\\\omega W(0)&0
		\end{pmatrix}
		+\BigOO{h^{\frac2{n+1}}\left(\log\frac1h\right)^{\delta_{n,1}}}{\Cc^2\to\Cc^2}.
	\end{equation*}
\end{rmk}

\subsection*{Transfer matrix for \ref{item:Caustic}: turning point} 

Since $(0,0)$ is a turning point, it is not possible to directly construct WKB solutions as above because $E_0-V_j=-V_j$ changes sign at $x=0$. Instead, we remove the turning point by rotating the phase space with the symplectomorphism $\kappa(y,\eta) := (\eta, -y)$ in order to have $\partial_\eta (\kappa^* p_j) (0,0) = V'_j(0) \neq 0$. The associated unitary metaplectic operator is $M_{\pi/2}:= e^{i\pi/4}\mc{F}_h$ where $\mc{F}_h$ is the semiclassical Fourier transform defined in \eqref{Eq:Fourier}. First, we construct WKB solutions for the rotated operator $M_{\pi/2}^{-1}P_{\rm Sc}M_{\pi/2}$ by solving the Eikonal and transport equations with a stationary phase argument, then we normalize them using Proposition \ref{prop:scalarmicsol} as above. We obtain
\begin{equation}\label{Eq:WKBSchrodingerCaustic}
    v_j^\dir(y) := e^{\frac{i}{h}\phi_j(y)} \sigma_j(y)(\mathbf{e}_j+\BigO{h})
\end{equation}
where $\phi_j(y) := \int_0^x V_j^{-1}(-y^2)\,dy$ and
$\sigma_j(y) := |V_j'(0)|^{1/2}|V_j'(V_j^{-1}(-y^2))|^{-1/2}$.
Next, we bring back the $v_j^\dir$ to normalized WKB\footnote{Rigorously, for \eqref{Eq:WKBRotatedSchrodinger} to rigorously define a WKB solution, we must replace $\sigma_j(y)$ by $\chi_\dir (y)\sigma_j(y)$ where $\chi_\dir(y)$ is a smooth cutoff supported in $(-\infty,0]$ for $\dir = \inc$ and $[0,+\infty)$ for $\dir = \out$.} solutions $w_j^\dir$ of the original problem by setting
\begin{equation}\label{Eq:WKBRotatedSchrodinger}
    w_j^\dir(x) := -iM_{\pi/2}v_j^\dir(x) := -iM_{\pi/2}\left[e^{\frac{i}{h}\phi_j(\cdot)} \sigma_j(\cdot)\right](x)(\mathbf{e}_j+\BigO{h}).
\end{equation}
The normalization constant $-i$ comes from the identities $\Theta(p_j) = -\pi/2$, $\Theta(\kappa^*p_j) = 0$ and
\begin{equation*}
    \frac{e^{-i\Theta(p_j)/2}}{\sqrt{2\pi h}}\int_{\Rr}e^{-\frac{x^2}{2h}}M_{\pi/2}v_j^\dir(x)dx=\frac{e^{i\pi/4}}{\sqrt{2\pi h}}\int_{\Rr}M_{\pi/2}\left(e^{-\frac{x^2}{2h}}\right)v_j^\dir(x)dx = i(1 + \BigO{h}).
\end{equation*}
We finally apply Theorem \ref{Th:LSM} after recalling that $m=2n$ as well as
\begin{equation*}
    \partial_{x,\xi} p_j(0,0) = (V_j'(0),0) \quad \text{and } H_{p_1}^np_2(0,\xi_0) = \frac{(-1)^{n}(2n)!}{n!}(V_1'(0))^{n}(V_2-V_1)^{(n)}(0) \quad \text{and } s = 1
\end{equation*}
and recover the result of \cite{AFH2} on the asymptotics of $T_{\rm Sc}(0,0)$:


\begin{prop}
	In the case \ref{item:Caustic}, the transfer matrix $T_{\rm Sc}(0,0)$ satisfies
	\begin{equation}\label{Eq:TransferMatrixSchrodingerCaustic}
		T_{\rm Sc}(0,0) =I_2-ih^{\frac1{2n+1}}
		\begin{pmatrix}
			0&\omega_1 W(0)\\ \omega_2 W(0)&0
		\end{pmatrix}
		+\BigOO{h^{\frac2{2n+1}}}{\Cc^2\to\Cc^2},
	\end{equation}
	as $h\to+0$ with 
	\begin{align}\label{Eq:OmegaSchrodingerCaustic}
		\omega_j = 2\left(\frac{|V_{\hat{j}}'(0)|}{|V_j'(0)|}\frac{(2n+1)n!}{\left|V_j'(0)\right|^{n}\left|(V_2-V_1)^{(n)}(0)\right|}\right)^{\frac{1}{2n+1}}\pmb\Gamma\left( \frac{2n+2}{2n+1}\right) \cos\left(\frac{\pi}{2(2n+1)}\right) \in \Rr.
	\end{align}
    where $\hat{j} = 1,2$ is such that $\{j,\hat{j}\} = \{1,2\}$.
\end{prop}


\subsection*{Comparison of transfer matrices}
Let us denote $w_j^\dir$ the WKB solutions we defined respectively in \eqref{Eq:WKBSchrodinger} and \eqref{Eq:WKBSchrodingerCaustic} on one hand, and $v_j^\dir$ those of respectively \cite{AFH} and \cite{AFH2} on the other hand. From Theorem \ref{Thm:ExistenceAndUniqueness} we know there exist $\gamma_j^\dir \in \Cc$, $j=1,2$, $\dir \in \{\inc,\out\}$, such that $v_j^\dir \equiv \gamma_j^\dir w_j^\dir$ near $\gamma_j^\dir$. A straightforward computation 
(somewhat similar as \eqref{eq:T-matrix-ori-nf})
gives:
\begin{equation}\label{Eq:TransferMatrixChangeWKB}
    T(w_1^\inc,w_2^\inc,w_1^\out,w_2^\out) = \diag(\gamma_1^\out\ \gamma_2^\out)T(v_1^\inc,v_2^\inc,v_1^\out,v_2^\out)\diag(\gamma_1^\inc\ \gamma_2^\inc)^{-1}.
\end{equation} 
In the case \ref{item:NoCaustic}, it is easy to see that $\gamma_j^\inc = E_0^{-1/4}c_j$ so that $(\gamma_1^\sharp)^{-1} \gamma_2^\flat$ cancels out with the term $|\partial_{(x,\xi)}p_2(0,0)|/|\partial_{(x,\xi)}p_1(0,0)|$ appearing in $\omega_1$. Therefore, our transfer matrix is the same as in \cite{AFH}. In the case \ref{item:Caustic}, a careful computation yields
$\gamma_j^\inc = \sqrt 2 |V_j'(0)|^{-1/2}$ and $\gamma_j^\out = i\sqrt 2|V_j'(0)|^{-1/2}$. This explains the factor $i$ appearing between our transfer matrix and that of \cite{AFH2}. 


\subsection{Notation} We fix some notations. We denote the set of natural numbers by $\Nn=\{0,1,2,\ldots\}$.
We write $D_{x_j}=i^{-1}\partial_{x_j}$. The function space $\mathscr{S}(\Rr^{n})$ denotes the set of all smooth, rapidly decreasing scalar functions on $\Rr^{n}$ and $\mathscr{S}'(\Rr^{n})$ denotes the set of all tempered distributions on $\Rr^{n}$. Let $C^\infty_b(\Rr^n)$ denote the symbol class of smooth functions with bounded derivatives on $\Rr^n$. For a function space $X\subset \mathscr{S}'(\Rr^n)$, let $\mc{M}_N(X)$ with $N\in\Nn\setminus\{0\}$ be the set of $N\times N$-matrix valued functions whose entries belong to $X$. For simplicity, for a $h$-dependent family of functions $a=(a_h)_{0<h\leq h_0}$ with $0<h_0\leq 1$ and a Fr\'echet space $X$, we write $a\in X$ if all the seminorms of $a_h$ in $X$ are uniformly bounded with respect to $0<h\leq h_0$.
For Banach spaces 
$X$ and $Y$,
$B(X,Y)$ denotes the set of all linear bounded operators from $X$ to $Y$. For a Banach space $X$, we denote the norm of $X$ by $\|\cdot\|_{X}$.  We say that a set $S\subset\Rr^2$ is \textit{compactly contained in} another set $T\subset \Rr^2$ and write $S\Subset T$ if the closure of $S$ is compact and contained in $T$. 

\subsection*{Acknowledgment}

The first author was supported by 
Grant-in-Aid for JSPS Fellows
Grant Number JP22KJ2364.
The second author was supported by the Otsuka Toshimi Scholarship Foundation and by the AY2024 Ritsumeikan University Graduate School Research Grant for Doctoral Students.
The third author was supported by Grant-in-Aid for Early-Career Scientists
Grant Number JP23K13004.

\section{Preliminaries}\label{S:Preliminaries}


\subsection{Quantization and symbolic calculus}\label{SubS:QuantizationSymbols}

We briefly recall some basic properties of semiclassical and microlocal analysis to fix the notation.
Please refer to the textbooks \cite{DiSj,Ma,Zwo} for more details. 

The Weyl quantization of a matrix-valued function $a=a(x,\xi)\in\mc{M}_N(C^\infty_b)$ is the semiclassical pseudo-differential operator $a^w(x,hD_x)$ defined by
\begin{equation}\label{Eq:Weyl}
    a^w(x,hD_x)u(x):=
    \frac{1}{(2\pi h)^n}\iint_{T^*\Rr^n}
    e^{\frac ih(x-y)\cdot \xi}
    a\left(\frac{x+y}2,\xi\right)u(y)dyd\xi,
\end{equation}
for vector-valued functions $u\in\mathscr{S}(\Rr^n;\Cc^N)$. The function $a$ is called the \textit{symbol} of the operator $a^w(x,hD_x)$. Let $\Psi_h$ denote the set of semiclassical pseudodifferential operators with scalar symbols in $C^{\infty}_b(\Rr^n;\Cc)$. By the Calder\'on-Vaillancourt theorem (\cite[Theorem 4.23]{Zwo}), each element of $\Psi_h$ extends to a bounded operator on $L^2(\Rr^n;\Cc^N)$. We also define the set $\Psi_h^{\mathrm{comp}}$ that consists of operators $A\in \Psi_h$ whose integral kernels are compactly supported and 
whose
wave front sets (defined in the next subsection) are compact. 

For scalar-valued symbols
$a,b\in C_b^\infty(T^*\Rr^n;\Cc)$, there exists a symbol $a\# b\in C_b^\infty(T^*\Rr^n;\Cc)$ such that $a^w(x,hD_x)b^w(x,hD_x)=(a\# b)^w(x,hD_x)$ and $a\# b=ab+\frac h{2i}\{a,b\}+\BigOO{h^2}{C_b^\infty}$, where $\{a,b\}=\partial_\xi a \cdot\partial_x b -\partial_xa \cdot\partial_\xi b$.

We now state some useful boundedness properties of pseudodifferential operators on $L^{\infty}$, used to study the integral operators $\Gamma_{\pm}$ that appear in Subsection \ref{SubS:simple} (see also Proposition \ref{prop:normbdd}).

\begin{lemma}\label{Lem:L-infty-bdd}
Let $A\in \Psi_h^{\mathrm{comp}}$. Then there exists $C>0$ (independent of $0<h\leq 1$) such that $\|A\|_{L^\infty\to L^\infty}, \|A\|_{W^{1,\infty}\to W^{1,\infty}} \leq C$ for $0<h\leq 1$. Here $W^{1,\infty}$ is the Sobolev space with the norm $\|u\|_{W^{1,\infty}}:=\|u\|_{L^{\infty}}+\sum_{j=1}^n\|D_{x_j}u\|_{L^{\infty}}$. 
\end{lemma}

\begin{proof}
We write $A=a^w(x,hD_x)$ with $a\in \ms{S}(\Rr^n)$.
First, we prove the uniform boundedness of $\|A\|_{L^{\infty}\to L^{\infty}}$.
The proof is essentially the same as \cite[Theorem 7.15]{Zwo} that deals with Fourier multipliers. The integral kernel of $A$ is given by $K_h(x,y)=\frac{1}{(2\pi h)^n}\int_{\Rr^n}a\left(\frac{x+y}{2},\xi\right)e^{\frac{i}{h}(x-y)\cdot \xi}d\xi$. An integration by parts shows $|K_h(x,y)|\leq Ch^{-n}(1+|x-y|/h)^{-n-1}$ with a constant $C>0$ independent of $x,y,h$. By Young's inequality, we have $\|Au\|_{L^{\infty}}\leq Ch^{-n}\|(1+|x|/h)^{-n-1}\|_{L^1} \|u\|_{L^{\infty}}= C\|(1+|x|)^{-n-1}\|_{L^1}\|u\|_{L^{\infty}} $, which proves the uniform boundedness of $\|A\|_{L^\infty\to L^\infty}$. 
Since $[D_{x_j},A]=-i(\pa_{\xi_j}a)^w(x,hD_x)$ and $\pa_{\xi_j}a\in\ms{S}(\Rr^n)$,
\begin{align*}
\|D_{x_j}Au\|_{L^{\infty}}\leq& \|AD_{x_j}u\|_{L^{\infty}}+\|[D_{x_j},A]u\|_{L^{\infty}}\\
\leq& \left(\|A\|_{L^{\infty}\to L^{\infty}}+\|(\pa_{x_j}a)^w(x,hD_x)\|_{L^{\infty}\to L^{\infty}} \right)\|u\|_{W^{1,\infty}}\leq C\|u\|_{W^{1,\infty}},
\end{align*}
which proves the uniform boundedness of $\|A\|_{W^{1,\infty}\to W^{1,\infty}}$.
\end{proof}

\subsection{Microlocalization and wavefront set}\label{SubsubS:WFh}

\begin{definition}\label{Def:MicrolocallySmall}
Let a family $f=\{f_h\}_{h\in(0,h_0]} \subset L^2(\Rr^n;\mathbb{C}^N)$ be \textbf{\textit{tempered} in $h$}, that is, there exists $k\ge0$ such that  $f_h=\BigOO{h^{-k}}{L^2}$ as $h\to 0$. 

\noindent$(i)$ We say $f$ is \textbf{\textit{microlocally infinitely small}} near a point $(x_0,\xi_0)\in T^*\Rr^n$ and write
\begin{equation*}
    f\equiv0\quad\text{m.l. near }(x_0,\xi_0),
\end{equation*} 
if there exists a function $\chi\in \mc{M}_N(C^\infty_b)$ such that $\chi(x_0,\xi_0)=1$ and $\left\|\chi^w(x,hD_x)f(h)\right\|_{L^2(\Rr^n)}=\mc{O}(h^\infty)$.
We also write $f\equiv0\,\,\text{m.l. on } \Omega$ if it is microlocally infinitely small near each point of $\Omega$. The \textbf{\textit{semiclassical wavefront set}} $\operatorname{WF}_h(f)\subset T^*\Rr^n$ is defined as the set of points where $f$ is not microlocally infinitely small. We say that $f$ is \textbf{\textit{microlocalized}} on a set $U$ whenever $\operatorname{WF}_h(f)\subset U$.

\noindent$(ii)$ Let $A = a^w(x,hD_x)$ be a pseudo-differential operator with a symbol $a\in \mc{M}_N(C^\infty_b)$. We say $f$ is a \textbf{\textit{microlocal solution}} to $Af=0$ on $\Omega\subset \Rr^{2n}$ if $Af\equiv 0$ m.l. on $\Omega$.

\noindent$(iii)$ We say that $f$ is \textbf{\textit{compactly microlocalized}} whenever there exists $\chi\in \mc{M}_N(C^\infty_c)$ independent of $h$ such that $(1 - \chi^w(x,hD_x))f = \BigOO{h^\infty}{\ms{S}(\Rr^n)}$.

\end{definition}

\begin{rmk}
If $f\equiv0$ m.l. near $(x_0,\xi_0)\in T^*\Rr^n$, then there exists a neighborhood $(x_0,\xi_0)\in U\subset T^*\Rr^n$ such that $\left\|\chi^w(x,hD_x)f(h)\right\|_{L^2(\Rr^n)}=\BigO{h^\infty}$ for any $\chi\in \mc{M}_N(C^\infty_c)$ whose support is contained in $U$ (see \cite[Theorem 8.13]{Zwo}).

\end{rmk}

\begin{definition}



\noindent$(i)$ 
We say a family $T=\{T_h\}_{h\in(0,h_0]}$ of continuous linear operators $T_h:\mathscr{S}\to \mathscr{S}'$ is \textbf{\textit{tempered}} if there exists $N\ge0$ such that $\left\|\mathscr{L}_{-N}T_h\mathscr{L}_{-N}\right\|=\BigOO{h^{-N}}{L^2\to L^2}$ as $h\to 0$, where $\mathscr{L}_{-N}:=(1+|x|^2)^{-N}(1+|hD_x|^2)^{-N}$.

\noindent$(ii)$ Let  $T=\{T_h\}_{h\in(0,h_0]}$ and $S=\{S_h\}_{h\in(0,h_0]}$ be tempered families of operators and $z$ and $w$ in $T^*\Rr^n$. We write $T\equiv S$ m.l. near $(z,w)$ if there exist neighborhoods $U$ and $V$ of $z$ and $w$ such that
\begin{equation*}
    B(T-S)A=\BigOO{h^\infty}{\mathscr{S}\to\mathscr{S}}.
\end{equation*} 
for any pair of pseudodifferential operators $A=a^w(x,hD_x)$ and $B=b^w(x,hD_x)$ with $\supp a\subset U'$ and $\supp b\subset V'$.

\noindent$(iii)$ We say that a tempered family $T=\{T(h)\}_{h\in(0,h_0]}$ of operators is \textbf{\textit{microlocally invertible}} near  
$z$ and $w$ in $T^*\Rr^n$ if there exist neighborhoods $U$ and $V$ of $z$ and $w$ and
a tempered family $S=\{S_h\}_{(0,h_0]}$ of operators such that
\begin{equation*}
    ST\equiv I\quad\text{m.l. near }(z,z)\quad\text{and}\quad
    TS\equiv I\quad\text{m.l. near }(w,w).
\end{equation*}
Then $S$ is called \textbf{\textit{a microlocal inverse}} near 
$(z,w)$
of $T$.
We denote $S$ simply by $T^{-1}$.

\end{definition}


\begin{lemma}\label{lem:Psicomp}
For any $A\in \Psi_h$ and any compact set $K\subset T^*\Rr^n$, there exists $A'\in \Psi_h^{\mathrm{comp}}$ such that $A\equiv A'$ m.l. on $K\times K$.
\end{lemma}

\begin{proof}
This follows from the the pseudolocal property of the pseudodifferential operators.
\end{proof}

\begin{prop}[Elliptic parametrix]\label{prop:elliptic}
Let $A=a^w(x,hD_x)$ be a pseudodifferential operator with a matrix-valued symbol $a\in \mc{M}_N(C_b^\infty)$.

\noindent$(i)$
If the inequality $\left|\det a\right|\ge c$ holds on $U$ with an $h$-independent $c>0$, 
then $A$ is microlocally invertible near $U\times U$.
Moreover, the symbol of each microlocal inverse of $A$ coincides modulo $\BigOO{h}{\mc{M}_N(C_b^\infty)}$ with $a^{-1}$ on $U$.

\noindent$(ii)$
Let $u$ be a microlocal solution near $U\subset T^*\Rr^n$ to the pseudodifferential equation $Au(x)=0$. 
The first statement implies that one has
\begin{equation*}
    \operatorname{WF}_h(u)\cap U\subset \{\det a=\BigO{h^\infty}\}.
\end{equation*}

\end{prop}

\begin{proof}
This follows from the usual parametrix construction (for example, see \cite[Theorem 4.29]{Zwo}). We omit the details.
\end{proof}

\begin{ex}\cite[Section 8.4, Examples]{Zwo}\label{ex:WF}
    One has $\WF_h(\sigma(x)e^{ih^{-1}\phi(x)}) \subset \{(x,\partial_x\phi(x));\,x\in\operatorname{supp}\sigma\}$ for $\sigma=\sigma(x)\in C^\infty(\Rr^n)\cap L^2(\Rr^n)$ and $\phi=\phi(x)\in C^\infty(\Rr^n;\Rr)$ (real-valued).
\end{ex}

\subsection{Fourier integral operators}
We review some properties of semiclassical Fourier integral operators.
We refer to \cite{Zwo,NZ} for proofs, details, and related topics.

\begin{definition}\label{def:FIO}
Let $\kappa$ be a \textbf{\textit{symplectomorphism}}
near $(0,0)$, that is to say a diffeomorphism between neighborhoods of $(0,0)\in T^*\Rr^n$ preserving the standard symplectic form $d\xi\wedge dx$ and satisfying $\kappa(0,0)=(0,0)$.

\noindent$(i)$ Let $F$ be a bounded operator on $L^2(\Rr^n)$.
We say that the operator $F$ \textbf{\textit{quantizes $\kappa$}} near $(0,0)$ if for all $a\in C_b^{\infty}(T^*\Rr^n)$, we have
\begin{align*}
a^w(x,hD_x)F=Fb^w(x,hD_x)\quad \text{m.l. near $((0,0),(0,0))$}
\end{align*}
where $b\in C_b^{\infty}(T^*\Rr^n)$ satisfies $b=\kappa^*a+\BigOO{h}{C_b^\infty}$.

\noindent$(ii)$  Assume that the projection 
\begin{equation}\label{eq:projdiffeo}
    \left\{(x,\xi,y,\eta)\in T^*\Rr^n\times T^*\Rr^n;\,\kappa(y,\eta)=(x,\xi)\right\}\ni(x,\xi,y,\eta)\mapsto (x,\eta)\in\Rr^n\times\Rr^n
\end{equation}
is a diffeomorphism near $(0,0)$\footnote{This condition is equivalent to the $n\times n$ block $\frac{\pa x}{\pa y}$ in the derivative $d\ka(0,0)$ being invertible, where $(x,\x)=\ka(y,\y)$.}.
A smooth function $\psi$ defined near $(0,0)\in \Rr^n\times \Rr^n$ is called a \textbf{\textit{generating function}}\footnote{A generating function always exists due to the assumption that the projection \eqref{eq:projdiffeo} is a diffeormorphsim near $(0,0)$ due to the implicit function theorem and Poincar\'e's lemma.} of $\kappa$ if for $(x,\eta)$ near $(0,0)$, 
\begin{equation}\label{eq:generating-fn}
    \kappa(\partial_\eta \psi(x,\eta),\eta)=(x,\partial_x \psi(x,\eta)),\quad
    \det \partial_x\partial_\eta \psi(x,\eta)\neq0,\quad
    \psi(0,0)=0.
\end{equation}
\noindent$(iii)$ Under the same assumption as $(ii)$, a bounded operator $F=F_{h}$ on $L^2(\Rr^n)$ is called a \textbf{\textit{Fourier integral operator}} associated with $\kappa$ \footnote{We can still define the notion of Fourier integral operators without the assumption on the projection \eqref{eq:projdiffeo}. However, the general definition is a bit complicated to state and not needed in this paper.} if $F$ is of the form
\begin{equation}\label{eq:def-FIO}
    F u(x)
    =\frac1{(2\pi h)^n}\iint_{\Rr^{2n}}e^{\frac ih (\psi(x,\eta)-y\cdot\eta)}\alpha(x,\eta)u(y)dyd\eta,
\end{equation}
with a function $\alpha(x,\eta)\in C_c^\infty(\Rr^n\times\Rr^n)$ supported near $(0,0)$ where \eqref{eq:generating-fn} holds and $\alpha\sim \sum_{j=0}^{\infty}h^j\alpha_j$. We say that $F$ is \textbf{microlocally unitary} at $((0,0),(0,0))$ whenever 
\begin{align}\label{eq:mlunitary}
\alpha_0(0,0)=\left|\partial_x\partial_\eta\psi(0,0))\right|^{1/2}
\end{align}
holds\footnote{As is stated in \cite[$(4.5)$]{NZ}, if $F^*F\equiv I$ m.l. near $((0,0),(0,0))$, then $|\alpha_0(0,0)|=\left|\partial_x\partial_\eta\psi(0,0))\right|^{1/2}$ holds. Here, we impose the stronger condition \eqref{eq:mlunitary} just for simplicity.}.
\end{definition}

\begin{prop}\label{Prop:FIO-conjugate}
Let $\kappa$ be a symplectomorphism near $(0,0)$. Assume that the projection \eqref{eq:projdiffeo} is a diffeomorphism near $(0,0)$. Then, the following holds:

\noindent$(i)$ $($Egorov's theorem$)$
Each Fourier integral operator $F$ with $\alpha_0(0,0)\neq 0$ quantizes $\kappa$ near $(0,0)$. Conversely, if $F$ is a bounded operator on $L^2(\Rr^n)$ quantizing $\kappa$ near $(0,0)$, then $F$ is a Fourier integral operator associated with $\kappa$.

\noindent$(ii)$ If a Fourier integral operator $F$ associated with $\kappa$ is microlocally unitary at $((0,0),(0,0))$ and $A=a^w(x,hD_x)\in \Psi_h$ satisfies $a(0,0)=1$, then we can find a Fourier integral operator $F'$ such that $FA\equiv F'$ m.l. near $((0,0),(0,0))$ and $F'$ is a Fourier integral operator associated with $\kappa$ which is microlocally unitary at $((0,0),(0,0))$.

\end{prop}

\begin{proof}
The proof is standard and the first statement is given around \cite[$(4.6)$]{NZ} (see also \cite[Lemma 3.4]{SjZw}). We omit the details.

\end{proof}

The following proposition is a modified version of \cite[Lemma 4.1]{NZ}, and provides a kind of composition formula for Fourier integral operators.

\begin{prop}\label{Prop:FIO-WKB}
Suppose $n=1$. Let $p,\phi,\tilde{\phi}\in C^{\infty}(T^*\Rr)$ be real-valued functions and $\kappa$ be a symplectomorphism near $(0,0)$ such that

\begin{itemize}
\item $\phi(0)=\tilde{\phi}(0)=\phi'(0)=\tilde{\phi}'(0)=p(0,0)=0$ and $\partial_{(x,\xi)}p(0,0)\neq 0$;

\item $p(x,\phi'(x))=(\kappa^*p)(x,\tilde{\phi}'(x))=0$ around $x=0$\footnote{These conditions with $\partial_{(x,\xi)}p(0,0)\neq 0$ imply $\partial_{\xi}p(0,0)\neq 0$ and $\partial_{\xi}(\kappa^*p)(0,0)\neq 0$. Conversely, if $\partial_{\xi}p(0,0)\neq 0$ and $\partial_{\xi}(\kappa^*p)(0,0)\neq 0$ are satisfied, then such phase functions $\phi,\tilde{\phi}$ exist due to the implicit function theorem and Poincar\'e's lemma.};

\item the projection \eqref{eq:projdiffeo} is a diffeomorphism near $(0,0)$.

\end{itemize}
Suppose that a bounded operator $F$ on $L^2(\Rr^n)$ is a Fourier integral operator associated with $\kappa$ and is microlocally unitary at $((0,0),(0,0))$ in the sense of Definition \ref{def:FIO} $(iii)$. 
Then, for $\tilde{\sigma} = \tilde{\sigma}_h \in C_b^{\infty}(\Rr)$ satisfying $\tilde{\sigma}(x)=\tilde{\sigma}_0(x)+\BigOO{h}{C_b^{\infty}}$ with $\tilde{\sigma}_0\in C_b^{\infty}(\Rr)$ independent of $h$, we have
    \begin{equation}\label{eq:Fkappaact}
        F (e^{\frac ih \tilde{\phi}(x)}\tilde{\sigma})(x)= e^{\frac ih \phi(x)}\sigma(x)
    \end{equation}
    with a function $\sigma=\sigma_h\in C_b^{\infty}(\Rr)$ satisfying $\sigma(x)=\sigma_0(x)+\BigOO{h}{C_b^{\infty}}$ with $\sigma_0\in C_b^{\infty}(\Rr)$ independent of $h$. Moreover,
    \begin{equation}\label{eq:WKBFIO}
\sigma_0(0)=e^{\frac{\pi i}{4}\upsilon}\left|\frac{\partial_\xi(\kappa^*p)(0,0)}{\partial_\xi p(0,0)}\right|^{1/2}\tilde{\sigma}_0(0).
    \end{equation}Here, the index $\upsilon$ is given by
\begin{equation*}
\upsilon=\operatorname{sgn}\partial_{(y,\eta)}^2(\tilde{\phi}(y)-y\cdot\eta+\psi(0,\eta))|_{(y,\eta)=(0,0)},
\end{equation*}
where $\psi$ is a generating function of $\kappa$.
\end{prop}

\begin{proof}
By the assumption, we can find neighborhoods $\Omega,\tilde{\Omega}$ of $0\in \Rr$ such that $\kappa(\tilde{\Lambda})=\Lambda$, where $\Lambda=\{(x,\phi'(x))\mid x\in \Omega\}$ and  $\tilde{\Lambda}=\{(x,\tilde{\phi}'(x))\mid x\in \tilde{\Omega}\}$. By virtue of \cite[Lemma 4.1]{NZ}, we can find $\sigma=\sigma_0+\BigOO{h}{C_b^{\infty}}$ satisfying \eqref{eq:Fkappaact} and
\begin{align}\label{eq:NZcal}
\sigma_0(x)=e^{\frac{\pi i}{4}\upsilon}\frac{\alpha_0(x,\phi'(g(x)))}{\left|\partial_x\partial_\eta\psi(x,\phi'(g(x)))\right|^{1/2}}\left|g'(x)\right|^{1/2}\tilde{\sigma}_0(g(x)),
\end{align}
where, $g(x)$ denotes the function defined near $x=0$ such that $\kappa(g(x),\tilde{\phi}'(g(x)))=(x,\phi'(x))$. The precise value of $\upsilon$ is not calculated in \cite[Lemma 4.1]{NZ}. We do not include this calculation here, as it can be easily done using 
a stationary phase argument.
Next, we claim that
\begin{align}\label{eq:gderp}
g'(0)=\frac{\partial_\xi(\kappa^*p)(0,0)}{\partial_\xi p(0,0)}
\end{align}
holds. Let $q=q(x,\xi)$ be a function defined near $(0,0)$ such that $\kappa^*q(x,\xi)=x$ and $q(0,0)=0$. Then, one has\footnote{By the definition of $g$, we have $q(x,\tilde{\phi}'(x))=q(\kappa(g(x),\ast))=(\kappa^*q)(g(x),\ast)=g(x)$.} $g(x)=q(x,\tilde{\phi}'(x))$, and
\begin{equation*}
g'(x)=\left(\partial_x q+\tilde{\phi}''(x)\partial_\xi q\right)(x,\phi'(x))
=\frac{\{p,q\}}{\partial_\xi p}(x,\tilde{\phi}'(x)),
\end{equation*}
where we use $(\pa_xp+\tilde{\phi}''(x)\pa_{\xi}p)(x,\tilde{\phi}'(x))=0$, which follows from $p(x,\tilde{\phi}'(x))=0$.
Then $\tilde{\phi}'(0)=0$ and $\{p,q\}=\{\kappa^*p,\kappa^*q\}=\partial_\xi(\kappa^*p)$ imply \eqref{eq:gderp}.

Now, we observe that $g(0)=0$ holds by $\kappa(0,0)=(0,0)$.
Since we assumed $F$ is microlocally unitary at $((0,0),(0,0))$, we have $\alpha_0(0,0)=\left|\partial_x\partial_\eta\psi(0,0))\right|^{1/2}$ (see Definition \ref{def:FIO}). Combining these with \eqref{eq:NZcal} and \eqref{eq:gderp}, we obtain \eqref{eq:WKBFIO}.

\end{proof}

\subsection{Metaplectic operators associated with rotations}\label{S:rotation}
Fourier integral operators associated with linear symplectic maps can be defined as unitary operators on $L^2$ and are called \textit{metaplectic operators}.
We recall some properties of a metaplectic operator $M_\theta:L^2(\Rr)\to L^2(\Rr)$ associated with a one-dimensional ($n=1$) linear symplectic map $\kappa_\theta:\Rr^2\to\Rr^2$ 
corresponding to a rotation:
\begin{equation}\label{eq:rotation}
    \kappa_\theta(x,\xi)=((\cos\theta)x+(\sin\theta)\xi,-(\sin\theta)x+(\cos\theta)\xi).
\end{equation}
The operator $M_\theta$ is uniquely characterized up to a multiplication by a constant of modulus one by its unitarity and by the exact form of the Egorov's theorem:
\begin{equation}\label{eq:ExEgorov}
    M_\theta^{-1}a^w(x,hD_x)M_\theta=(\kappa_\theta^*a)^w(x,hD_x)\quad
    \text{for any } a\in C_b^\infty(\Rr^2).
\end{equation}
In this manuscript, we fix it as an integral operator
\begin{equation*}
    M_\theta u(x)=
    \frac{\left|\cos\theta\right|^{-1/2}}{2\pi h}\int_{\Rr^2}\exp\left(\frac ih\left(-\frac{\tan\theta}2\left(x^2-2\frac{x\eta}{\sin\theta}+\eta^2\right)-y\eta\right)\right)u(y)dyd\eta,
\end{equation*}
for $|\theta|<\pi/2$, and $M_{\pi/2}=e^{i\pi/4}\mc{F}_h$, where 
$\mc{F}_h$ is the semiclassical Fourier transform defined by
\begin{equation}\label{Eq:Fourier}
    \mc{F}_hu(x)=\frac1{\sqrt{2\pi h}}\int_{\Rr}e^{-\frac ih xy}u(y)dy.
\end{equation}

\begin{lemma}\label{lem:inv-metaplectic-normal}
For $-\pi/2 < \theta\le \pi/2$, one has
    \begin{equation*}
    M_\theta e^{-\frac{\cdot^2}{2h}}(x)=
    e^{i\frac{\theta}{2}}e^{-\frac{x^2}{2h}}.
    \end{equation*}
\end{lemma}

\begin{proof}
First, we recall that one has $\mc{F}_he^{-\frac{\cdot^2}{2h}}(\eta)=e^{-\frac{\eta^2}{2h}}$.
This implies 
\begin{align*}
    M_\theta e^{-\frac{\cdot^2}{2h}}(x)
    &=\frac{|\cos\theta|^{-1/2}}{\sqrt{2\pi h}}\int_{\Rr}\exp\left(\frac {-e^{i\theta}\eta^2+i2x\eta-i(\sin\theta) x^2}{2h\cos\theta} \right)d\eta\\
    &=e^{-\frac{x^2}{2h}}\frac{|\cos\theta|^{-1/2}}{\sqrt{2\pi h}}
    \int_{\Rr}\exp\left(\frac {-e^{i\theta}(\eta-ie^{-i\theta}x)^2}{2h\cos\theta} \right)d\eta.
\end{align*}
Then by the change of variable $\xi=e^{i\theta/2}|\cos\theta|^{-1/2}(\eta-ie^{-i\theta}x)$, and by the Cauchy theorem, we obtain the required formula.
\end{proof}

\subsection{Translation only in $\xi$-variable}
As another special case of symplectic maps, consider
\begin{equation*}
    \kappa_\phi(x,\xi)=(x,\xi+\partial_x \phi(x))
\end{equation*}
for a real-valued function $\phi\in C^\infty_c(\Rr^n;\Rr)$.
In this case, the multiplication by $e^{i\phi/h}$ is a Fourier integral operator which quantizes $\kappa_\phi$.
\begin{prop}\label{Prop:conjugatePDO}
    Let $A=a^w(x,hD_x)$ be a pseudodifferential operator with $a\in \mathscr{S}(T^*\Rr^n)$. 
    One has 
    \begin{equation*}
        e^{-i\phi/h}Ae^{i\phi/h}=
        a_\phi^w(x,hD_x)+\BigOO{h^\infty}{\mathscr{S}'\to\mathscr{S}},
    \end{equation*}
    with $a_\phi=\kappa_\phi^*a+\BigOO{h}{C_b^\infty}$.
    Moreover, one has
    \begin{equation*}
        \operatorname{supp}a_\phi\subset\kappa_\phi^{-1}(\operatorname{supp}\,a).
    \end{equation*}
\end{prop}

According to the above proposition, 
a WKB state remains a WKB state under the action of a pseudodifferential operator
that is, for a polynomially bounded $\sigma\in C_b^\infty(\Rr^n)$ and a symbol $a\in \mathscr{S}(T^*\Rr^n)$, one has
\begin{equation}\label{eq:PseuDO-WKB}
    a^w(x,hD_x)\left(\sigma(x)e^{i\phi (x)/h}\right) 
    = 
    \left(a_\phi^w(x,hD_x)\sigma(x)\right)e^{i\phi(x)/h}+\BigOO{h^\infty}{\mathscr{S}}.
\end{equation}

\subsection{Microlocal solutions for the scalar equations away from the crossing point} 

In this subsection, we review some basic properties of the scalar equations
\begin{align}\label{eq:Scalareq}
P_jv_j\equiv 0
\end{align}
where $P_j=p_j^w(x,hD_x)$ and $p_j$ satisfies Assumption \ref{C:RealPrincipal}.

The first and second statements of the following proposition is a version of our main theorem for the scalar equation. The third statement 
tells 
how we choose a constant to get a normalized solution in the sense of Definition \ref{def:normalized} under the condition $\partial_{\xi}p_j(0,0)\neq 0$. The fourth one 
ensures that the normalization condition is invariant under the action of metaplectic transforms with positive small angles, which will be useful in deducing the connection formula when $\partial_{\xi}p_j(0,0)=0$. 

\begin{prop}\label{prop:scalarmicsol}
Fix $j=1,2$. 

\noindent$(i)$ The space of microlocal solutions to the equation \eqref{eq:Scalareq} on each connected small neighborhood of $(0,0)$ is one-dimensional.\footnote{
More precisely, there exists a microlocal solution $v_j$ to \eqref{eq:Scalareq} such that $v_j\not\equiv0$ near $(0,0)$, and 
for any microlocal solution $\til{v}_j$, there exists a constant $k$ such that $\til{v}_j\equiv kv_j$ m.l. near $(0,0)$.}

\noindent$(ii)$ If $\pa_{\xi}p_j(0,0)\neq 0$, then for each $c_j\in \Rr$ independent of $h$, the unique microlocal solution $v_j$ (m.l. near $(0,0)$) to \eqref{eq:Scalareq} with $v_j(0)=c_j+\BigO{h^{\infty}}$ is of the form
\begin{align}\label{eq:scalarunisol}
v_j(x)=e^{\frac ih\phi_j(x)}\sigma_j(x,h),\quad \sigma_j(0,h)=c_j,
\end{align}
where $\phi_j$ is the unique solution to the Eikonal equation
\begin{equation}\label{eq:p_jeikonal}
    p_j(x,\phi'_j(x))=0,\qquad\phi_j(0)=0
\end{equation}
and $\sigma_j(x,h)\in C_b^{\infty}(\Rr)$ can be constructed by solving the transport equations and admits the asymptotic expansion $\sigma_j(x,h)\sim\sum_{l\ge0}h^l\sigma_{j,l}(x)$.

\noindent$(iii)$ The microlocal solution $v_j$ given in \eqref{eq:scalarunisol} is normalized in the sense of Definition \ref{def:normalized} $(i)$ if and only if
\begin{align}\label{eq:scalarnorconst}
c_j=\sqrt{\frac{\left|\partial_{(x,\xi)} p_j(0,0)\right|}{\left|\partial_\xi p_j(0,0)\right|}}.
\end{align}

\noindent$(iv)$ The metaplectic operator $M_\theta$ maps a normalized microlocal solution $v_{j,\theta}$ near $(0,0)$ of $M_\theta^{-1}P_jM_{\theta}v=0$ to that of \eqref{eq:Scalareq}, for small $\theta>0$.
\end{prop}

\begin{rmk}
    The function
    \begin{equation*}
        -\arctan\left(\frac{\partial_x p_j(\exp(tH_{p_j})(0,0))}{\partial_\xi p_j(\exp(tH_{p_j})(0,0))}\right)
    \end{equation*}
    has a discontinuity at $t=0$ if $t=0$ is a simple zero of $(\partial_\xi p_j)(\exp(tH_{p_j})(0,0))=0$.
    Then, the normalization constants $c_{j,\pm\varepsilon}$ for the WKB solution $v_j$ at $\exp(\pm\varepsilon H_{p_j})(0,0)$ 
    differ by
    $\pi/2$ in the limit $\varepsilon\to+0$: 
    \begin{equation*}
        \lim_{\varepsilon\to+0}\left|c_{j,+\varepsilon}-c_{j,-\varepsilon}\right|=\frac\pi2.
    \end{equation*}
    This gives an elementary explanation how the  Maslov index appears.
\end{rmk}

\begin{rmk}\label{Rem:index-theta}
    In general, the normalization is not invariant by the action of $M_\theta$.
    If $v_{j,\theta}$ is a normalized microlocal solution to $M_\theta^{-1}P_jM_\theta v_{j,\theta}=0$ with $-\pi/2<\theta<\pi/2$, 
    then the function
    \begin{equation*}
        v_j:=e^{-\frac {i\pi}2 \nu}M_{\theta}v_{j,\theta}
    \end{equation*}
    is a normalized microlocal solution of $P_jv_j=0$, where $\nu$ is given by \eqref{eq:index-nu-rotation}.
\end{rmk}

The proof of $(i)$ and $(ii)$ is well-known and we omit to write it here (for the uniqueness statement, see \cite[Lemme 18]{CdVPa}). In the following, we prove the parts $(iii)$ and $(iv)$.

\begin{proof}[Proof of Proposition \ref{prop:scalarmicsol} $(iii)$]
Let $v_j(x)=e^{\frac ih\phi_j(x)}\sigma_{j}(x)$ be the microlocal solution of the form \eqref{eq:scalarunisol}.
For any $h$-independent $\delta>0$, the integral 
outside a
$\delta$-neighborhood of $x=0$ is exponentially small with respect to $h\to+0$:
\begin{equation}\label{eq:Normal-away}
    \left|\int_{|x|\ge\delta}e^{-\frac{x^2}{2h}}e^{\frac ih\phi_j(x)}\sigma_{j}(x)dx\right|\le
    Ce^{-\delta^2/2h}.
\end{equation}

On the other hand, if $\delta>0$ is sufficiently small, the change of the variable $x^2-2i\phi_j(x)=y^2$ is valid for $x\in[-\delta,+\delta]$. 
Note that one has $\phi_j(x)=\phi_j''(0)x^2/2+\BigO{x^3}$ as $x\to0$, and $y= (1-i\phi_j''(0))^{1/2}x(1+\BigO{x})$.
Then it follows from the expansion  $\sigma_{j}(x)=c_j+\BigO{x}+\BigO{h}$ near $x=0$ that
\begin{align}\label{eq:Normal-near}
    &\frac1{\sqrt{2\pi h}}\int_{|x|\le \delta}e^{-\frac{x^2}{2h}}e^{\frac ih\phi_j(x)}\sigma_j(x)dx=\frac1{\sqrt{2\pi h}}\int_{|x|\le \delta}e^{-\frac{x^2}{2h}}e^{\frac ih\phi_j(x)}(c_j+\BigO{x}+\BigO{h})dx
    \\&=\frac{(1-i\phi_j''(0))^{-1/2}}{\sqrt{2\pi h}}\int_{\gamma_\delta}e^{-\frac{y^2}{2h}}(c_j+\BigO{y}+\BigO{h})
    =\frac{c_j+\BigO{h}}{(1-i\phi_j''(0))^{1/2}}\nonumber\\
    &=\left|\frac{\partial_\xi p_j(0,0)}{\partial_{(x,\xi)} p_j(0,0)}\right|^{\frac{1}{2}}
    c_je^{i\Theta(p_j)/2}+\BigO{h}\nonumber,
\end{align}
with a contour $\gamma_\delta$ in the complex plane, where we have applied the Cauchy theorem and the identities
\begin{align*}
    &\left|1-i\phi_j''(0)\right|^2=1+\left|\phi_j''(0)\right|^2=\left|\frac{\partial_{(x,\xi)} p_j(0,0)}{\partial_\xi p_j(0,0)}\right|^2,
    \\&\operatorname{arg}\left((1-i\phi_j''(0))^{-1/2}\right)=\frac12 \arctan \phi_j''(0)
    =\frac12 \arctan\left(\frac{\partial_x p_j(0,0)}{\partial_\xi p_j(0,0)}\right)
    =-\frac{\Theta(p_j)}{2},
\end{align*}
deduced from \eqref{eq:scalarunisol}. Combining \eqref{eq:Normal-near} with \eqref{eq:Normal-away}, we obtain the claim in $(iii)$.
\end{proof}

\begin{proof}[Proof of Proposition \ref{prop:scalarmicsol} $(iv)$]

Let $v_{j,\theta}$ satisfy the normalization condition \eqref{eq:FBI-normal} and $M_\theta^{-1}P_jM_\theta v_{j,\theta}\equiv0$ microlocally near the origin.
Then Lemma~\ref{lem:inv-metaplectic-normal} together with the unitarity of $M_\theta$ and $M_\theta^*=M_{-\theta}$ shows for $-\pi/2<\theta<\pi/2$ the identity
\begin{align*}
    \frac{e^{-\frac i2\Theta(p_j)}}{\sqrt{2\pi h}}\int_{\Rr}e^{-\frac{x^2}{2h}}M_\theta v_{j,\theta}(x)dx
    =\frac{e^{-\frac i2\Theta(p_j)}}{\sqrt{2\pi h}}\int_{\Rr}M_{-\theta}e^{-\frac{\cdot^2}{2h}}(x)v_{j,\theta}(x)dx
    =e^{\frac i2(\Theta(p_{j,\theta})-\Theta(p_j)-\theta)}.
\end{align*}
By definition of $\Theta$, one has the identity (see FIGURE \ref{Fig:rotation}) 
\begin{equation}\label{eq:ind-theta}
    \Theta(p_{j,\theta})=\Theta(p_j)+\theta+\pi\nu,
\end{equation}
with the index $\nu=\nu(p_j,\theta)\in\{0,\pm1\}$ determined by
    \begin{equation}\label{eq:index-nu-rotation}
        \Theta(p_j)+\theta+\pi\nu\in\left[-\frac\pi2,\frac \pi 2\right).
    \end{equation}
The proposition follows from the left-continuity with respect to $\theta$ of $\Theta(p_{j,\theta})$.
\end{proof}

\begin{figure}[h]
    \begin{minipage}{\linewidth}
    {\centering
    \includegraphics[width=\linewidth,page=2]{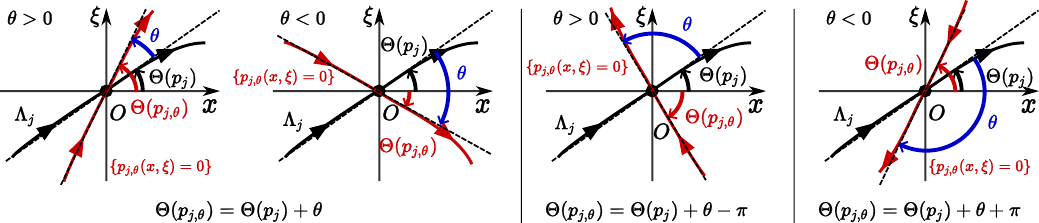}
    \caption{Rotation of the symbol $p_j$ by $\kappa_\theta$ for various $\theta$ and the relation between $\Theta(p_j)$, $\Theta(p_{j,\theta})$, and $\theta$ for each case (see \eqref{eq:index-nu-rotation}).}\label{Fig:rotation}
    }
    \end{minipage}
\end{figure}

\subsection{Normalized microlocal solutions the original equation}\label{SubsubS:WKBOriginal}

In this subsection, we discuss the behavior of microlocal solutions to the original equation
\begin{align}\label{eq:Mateq}
\mathscr{P}u\equiv 0
\end{align}
\textbf{away from the crossing point $(0,0)\in T^*\Rr$}. 
The following proposition is an analog of Proposition \ref{prop:scalarmicsol} for the scalar equation.

\begin{prop}\label{Prop:matrixmicsol}
$(i)$ For each $j=1,2$ and $\dir\in \{\inc,\out\}$, the space of microlocal solutions to the equation \eqref{eq:Mateq} on $\gamma_j^{\dir}$ is one-dimensional, where $\gamma_j^{\dir}$ is introduced in Condition \ref{C:ContactOrder}.
\noindent$(ii)$ Suppose that
\begin{equation}\label{eq:NoCaustic}
\partial_\xi p_1(0,0)\neq0\quad\text{and}\quad \partial_\xi p_2(0,0)\neq0
\end{equation}
hold. Each microlocal solution $w_j^{\dir}$ to \eqref{eq:Mateq} m.l. on $\gamma_j^{\dir}$ is given by
\begin{equation}\label{eq:matrixunisol}
w_j^\dir=e^{i\phi_j/h}\begin{pmatrix}
\sigma_{j,1}\\\sigma_{j,2}
\end{pmatrix},\quad \mathrm{WF}_h(w_j^{\dir})\subset \gamma_j^{\dir},
\end{equation}
where $\phi_j$ is the same as \eqref{eq:p_jeikonal} and $\sigma_{j,k}(x,h)\in C_b^{\infty}(\Rr\setminus \{0\})$ can be constructed by solving the transport equations and admits the asymptotic expansion $\sigma_{j,k}(x,h)\sim\sum_{l\ge0}h^l\sigma_{j,k,l}(x)$ with $\sigma_{j,3-j,0}=0$. Moreover, the principal part $\sigma_{j,j,0}$ is smooth at $x=0$ (see \cite[Remark 5.2]{AF}).
\noindent$(iii)$ The microlocal solution $w_j$ given in \eqref{eq:matrixunisol} is normalized in the sense of Definition \ref{def:normalized} $(ii)$ if and only if we choose $\sigma_{j,j,0}(0)$ as \eqref{eq:scalarnorconst}.
\noindent$(iv)$ The metaplectic operator $M_\theta$ maps a normalized microlocal solution $w_{j,\theta}$ near $(0,0)$ of $M_{\theta}^{-1}\mathscr{P}M_{\theta}u=0$ to that of \eqref{eq:Mateq}, for small $\theta>0$.
\end{prop}

\begin{rmk}
The amplitudes $\sigma_{j,k}$ are singular at $x=0$, which corresponds to the crossing point of $p_1$ and $p_2$. In other words, the usual WKB construction does not work at $x=0$. 

\end{rmk}

The proof of $(i)$ and $(ii)$ is more or less well-known. In fact, $(i)$ and $(ii)$ are consequences of \cite[Appendix B.2 and Proposition 5.1]{AF} whenever \eqref{eq:NoCaustic} holds. When condition \eqref{eq:NoCaustic} is violated, we choose $\theta\in \Rr$ such that $\kappa_{\theta}^*p_1$ and $\kappa_{\theta}^*p_2$ satisfy \eqref{eq:NoCaustic}. By virtue of \eqref{eq:rotation} and \eqref{eq:ExEgorov}, the problem can be reduced to the case where \eqref{eq:NoCaustic} holds. We omit the details. The part $(iii)$ directly follows from Proposition \ref{prop:scalarmicsol} $(iii)$.

Let us consider the case 
where
the operator $\mathscr{P}$ is of the following form:
\begin{equation*}
    \mathscr{P}=P_{\rm nf}:=\begin{pmatrix}
        hD_x&hR_1\\hR_2&hD_x-f(x)
    \end{pmatrix},
\end{equation*}
where $f\in C_b^\infty(\Rr;\Rr)$ be a smooth, real-valued function vanishing at $x=0$,
and $R_1$, $R_2\in\Psi_h^{\operatorname{comp}}$ are pseudo-differential operators.
In this case, the functions $\phi_j$ and $v_j$ are explicitly written as
\begin{equation*}
    \phi_1(x)=0,\quad
    \phi_2(x)=F(x)=\int_0^x f(y)dy,\quad
    v_1(x)=1,\quad
    v_2(x)=e^{\frac ihF(x)}(1+\BigO{h})
\end{equation*}
However, for simplicity, we use the WKB solutions $\tilde{v}_j(x)$ normalized by $|\tilde{v}_j(0)|=1$ instead of $v_j(x)$.
The normalized solutions are given by
\begin{equation*}
    \tilde{v}_1(x)=v_1(x)=1,\quad
    \tilde{v}_2(x)=e^{\frac ihF(x)}.
\end{equation*}
Then the WKB solutions $\tilde{w}_j^\dir$ to the equation $\scP w=0$ associated with $\tilde{v}_j$ admits the asymptotic behavior
\begin{equation}\label{eq:WKB-basis-nf}
    \tilde{w}_1^\dir(x)=\begin{pmatrix}1+\BigOO{h}{C_b^\infty}\\\BigOO{h}{C_b^\infty}\end{pmatrix},\qquad
    \tilde{w}_2^\dir(x)=e^{\frac ihF(x)}\begin{pmatrix}\BigOO{h}{C_b^\infty}\\1+\BigOO{h}{C_b^\infty}\end{pmatrix}\qquad(\dir=\inc,\out).
\end{equation}
The construction of such microlocal solutions is 
detailed 
in Appendix~\ref{App:WKB}.

\subsection{Outline of the proof of Theorems \ref{Thm:ExistenceAndUniqueness} and \ref{Th:LSM}}\label{SubS:OutlineProof}
In this manuscript, we prove our main results in the following way.
In Section~\ref{S:NormalForms}, we show the existence of a Fourier integral operator $\mathcal{F}$ which reduces the original equation $\mathscr{P}u=0$ to the 
reduced equation
$P_{\mathrm{nf}}\mathcal{F}^*u=0$
microlocally near $(0,0)$.
Then in Section~\ref{S:CFReduced}, we prove our main theorems for the 
reduced
operator.
Finally in Section~\ref{S:Correspondence}, we study the correspondence between the transfer matrices for the original operator and for the 
reduced
operator.

In more detail, the normal form reduction of Section~\ref{S:NormalForms} is done by a classical argument.
The symplectomorphism $\kappa$ is taken so that $\kappa^*p_1(x,\xi)=\xi$.
Then we can factorize $\kappa^*p_2=a_0(x,\xi)(\xi-f(x))$ with an elliptic factor $a_0$ by the implicit function theorem, which allows us to regard its quantization as a simple first order operator.
By the standard quantization procedure (\cite[Section 11.2]{Zwo}), we can find Fourier integral operators $F_1$ and $F_2$ which both quantize $\kappa$ such that $F_j^*P_jF_j$ for $j=1,2$ are of simple forms. The operator $\mathcal{F}$ is defined by $\mathcal{F}=\operatorname{diag}(F_1\ F_2)$.

The proof in Section~\ref{S:CFReduced} of the  existence and uniqueness of the microlocal solutions for given 
microlocal Cauchy data 
and the asymptotics of the transfer matrix 
for the reduced operator
is done by a microlocal method.
This differs from the use of exact solutions in previous studies, especially for the tangential case \cite{AFH,Higuchi}.
The main ingredient of this section is the estimate of an integral operator 
where the phase function of the integrand has a degenerate stationary point.

To compare the transfer matrices, we study in Section~\ref{S:Correspondence} the linear relationship between  $u$ and $\mathcal{F}v$, where $u$ is a microlocal solution to the original equation $\mathscr{P}u=0$ and $v$ is that to the 
reduced equation
$Pv=0$.
To see that, we study the action of the Fourier integral operator on WKB states. 
We also write the computation of the asymptotic behavior shown in Theorem~\ref{Th:LSM} at the end of this section.

\section{Reduction to a matrix normal form}\label{S:NormalForms}

In this section, we construct a normal form of our original operator $\ms{P}$ introduced in Subsection \ref{SubS:first}.

\begin{thm}\label{Thm:Matrix-NF}
Assume that Conditions~\ref{C:RealPrincipal} and \ref{C:ContactOrder} are true.
Then, there exist 
\begin{itemize}
\item a symplectomorphism $\kappa$ near $(0,0)$ (see Definition \ref{def:FIO});

\item pseudodifferential operators $A=a^w(x,hD_x)\in \Psi_h$ and $R_j=r_j^w(x,hD_x)\in \Psi_h^{\mathrm{comp}}$ for $j=1,2$ with $a,r_1,r_2\in C_b^\infty(\Rr^2)$;

\item a real-valued function $f\in C_b^\infty(\Rr;\Rr)$;

\item a family $\mc{F}=\{\mc{F}_h\}_{h\in(0,h_0]}=\operatorname{diag}(F_1\ F_2)$ of $L^2$-bounded operators microlocally invertible near $((0,0),(0,0))$,
\end{itemize}
such that the following holds:

\noindent$(i)$ We have
\begin{equation}
        \mc{F}^{-1}
        \mathscr{P}
        \mc{F}
        \equiv \operatorname{diag}(1\ A)
        \begin{pmatrix}
            hD_x&hR_1\\hR_2&hD_x-f(x)
        \end{pmatrix}\quad \text{m.l. near $((0,0),(0,0))$}.
    \end{equation}

\noindent$(ii)$ The operator $A$ is microlocally invertible near $((0,0),(0,0))$. Both operators $F_1$ and $F_2$ quantize $\kappa$ and are microlocally unitary at $((0,0),(0,0))$ in the sense of Definition \ref{def:FIO}.

\noindent$(iii)$ $(\kappa^*p_1)(x,\xi)=\xi$ around $(0,0)$, $\partial_{\xi}(\kappa^*p_2)(0,0)\neq 0$ and
\begin{align}\label{eq:Not-changed}
\partial_{\xi}p_j(0,0)\neq 0\quad (j=1,2)\Rightarrow \frac{\partial_\xi p_1(0,0)\partial_\xi p_2(0,0)}{\partial_\xi(\kappa^*p_1)(0,0)\partial_\xi(\kappa^*p_2)(0,0)}>0.
\end{align}

\noindent$(iv)$ Setting $c=(\partial_{\xi}(\kappa^*p_2)(0,0))^{-1}$, we have
    \begin{align*} 
    r_1(0,0)=q_1(0,0)+O(h),\quad r_2(0,0)=cq_2(0,0)+O(h)  
    \end{align*}
    and
    \begin{equation}\label{Eq:Derivative-f}
            f^{(k)}(0)=0\quad(0\le k\le m-1),\quad
            f^{(m)}(0)=-c(H_{p_1}^mp_2)(0,0).
    \end{equation}

\noindent$(v)$ In the tangential case $m\geq 2$, the constant $c$ is independent of $\kappa$ and calculated explicitly as
\begin{align*}
c
=s\frac{|\partial_{(x,\xi)} p_1(0,0)|}{|\partial_{(x,\xi)} p_2(0,0)|},\quad
    s=\operatorname{sgn}(\partial_{(x,\xi)} p_1(0,0),\partial_{(x,\xi)} p_2(0,0))_{\Rr^2}
    =\operatorname{sgn}(H_{p_1}(0,0),H_{p_2}(0,0))_{\Rr^2}.
\end{align*}

\end{thm}

\begin{rmk}
In the transversal case, the constant $c$ does depend on the choice of $\kappa$.
\end{rmk}

This theorem is proven in two steps. 
We first find a symplectomorphism $\kappa$ which reduces the symbols into a normal form (Lemma~\ref{Lemma:ClassicalNF}).
We then construct a suitable quantization $\mc{F}$ of $\kappa$ in the second step.
Lemma~\ref{Lemma:QuantumscalarNF} is the crucial part in the second step.

\begin{lemma}\label{Lemma:ClassicalNF}
There exist a locally elliptic real-valued symbol $a_0\in C_b^\infty(\Rr^2)$, a symplectomorphism $\kappa$ near $(0,0)$, and a real-valued function $f\in C_b^\infty(\Rr;\Rr)$ satisfying \eqref{Eq:Derivative-f} and the properties in Theorem~\ref{Thm:Matrix-NF} $(iii)$ such that 
\begin{equation}
        \left\{
        \begin{aligned}
            &(\kappa^*p_1)(x,\xi)=\xi\\
            &(\kappa^*p_2)(x,\xi)=a_0(x,\xi)(\xi-f(x)).
        \end{aligned}\right.
\end{equation}
Moreover, we have $a_0(0,0)=\partial_{\xi}(\kappa^*p_2)(0,0)$ and the constant $c$ satisfies the property stated in Theorem \ref{Thm:Matrix-NF} $(v)$. 

\end{lemma}

\begin{proof}

\underline{Step 1}: First, we construct a symplectomorphism $\kappa$ near $(0,0)$ satisfying the properties in Theorem~\ref{Thm:Matrix-NF} $(iii)$.
By virtue of Darboux's theorem (\cite[Theorem 12.1]{Zwo} with $A = \{1\}$ and $B = \emptyset$), there exists a symplectomorphism $\kappa_0$ preserving $(0,0)$ such that $(\kappa_0^*p_1)(x,\xi) = \xi$. 

\noindent$(1)$ In the tangential case, Condition~\ref{C:ContactOrder} implies $\{p_1,p_2\}(0,0)= 0$ and hence
\begin{align*}
\partial_{x}(\kappa_0^*p_2)(0,0)=\{(\kappa_0^*p_1),(\kappa_0^*p_2) \}(0,0)=0,
\end{align*}
where we use the fact that $\kappa_0$ preserves the Poisson bracket (see Jacobi's formula in \cite[Theorem 2.10]{Zwo}). Then we obtain $\partial_{\xi}(\kappa_0^*p_2)(0,0)\neq 0$ from Condition~\ref{C:RealPrincipal}. Moreover, the condition \eqref{eq:Not-changed} is automatically satisfied in the tangential case.
Thus we can take $\kappa:=\kappa_0$.

\noindent$(2)$ In the transversal case, it may happen that $\partial_{\xi}(\kappa_0^*p_2)(0,0)=0$ (for instance, one can consider $(\kappa_0^*p_2)(x,\xi)=x$) or \eqref{eq:Not-changed} may not be satisfied. 
In this case, we obtain the desired $\kappa$ by composing $\kappa_0$ with a linear symplectomorphism $\kappa_{\lambda}$. 
More precisely, we define $\kappa_{\lambda}(x,\xi)=(x+\lambda\xi,\xi)$ and put $\kappa:=\kappa_0\circ \kappa_{\lambda}$, where $\lambda\in \Rr$ is determined below. It is easy to check that the mapping $\kappa$ is a symplectomorphism preserving $(0,0)$, satisfying $(\kappa^*p_1)(x,\xi)=\xi$ and
\begin{align*}
\partial_{\xi}(\kappa^*p_2)(0,0)=\partial_{x}(\kappa_0^*p_2)(0,0)+\partial_{\xi}(\kappa_0^*p_2)(0,0)=\lambda\{p_1,p_2\}(0,0)+\partial_\xi(\kappa_0^*p_2)(0,0).
\end{align*}
Since $\{p_1,p_2\}(0,0)\neq 0$, we can find $\lambda\in \Rr$ such that $\partial_{\xi}(\kappa^*p_2)(0,0)\neq 0$ and \eqref{eq:Not-changed} (by choosing a sufficiently large $\lambda$ with the appropriate sign).

\underline{Step 2}:
According to the implicit function theorem with $\partial_{\xi}(\kappa^*p_2)(0,0)\neq 0$, there exists a globally defined, real-valued, smooth function $f$ such that, in a neighborhood of $(0,0)$,
\[
(x,\xi) \in \{\kappa^*p_2 = 0\} \iff \xi = f(x).
\]
On the other hand, Taylor's theorem at $(x,\xi)=(x,f(x))$ implies 
\[
(\kappa^*p_2)(x,\xi) = a_0(x,\xi)(\xi - f(x)),\quad \text{where} \,\, a_0(x,\xi):= \int_0^1 \partial_\xi (\kappa^*p_2)(x,f(x) + t(\xi - f(x)))dt. 
\]
Since $(0,0)\in \Lambda_2$, we have $f(0) = 0$ and hence $a_0(0,0) = \partial_\xi (\kappa^*p_2)(0,0) \neq 0$.

\underline{Step 3}: 
We prove that the function $f$ satisfies \eqref{Eq:Derivative-f}. 
To this end, we observe $\partial_x^{k}(\kappa^*p_2)=H_{p_1}^kp_2$ due to Jacobi's formula \cite[Theorem 2.10]{Zwo} with $(\kappa^*p_1)(x,\xi)=\xi$. Then Condition \ref{C:ContactOrder} implies
\begin{align}\label{Eq:tildep_2}
(\partial_x^{k}(\kappa^*p_2))(0,0)=0\quad (k\in \{1,\hdots,m-1\}),\quad (\partial_x^{m}(\kappa^*p_2))(0,0)=(H_{p_1}^mp_2)(0,0)\neq 0.
\end{align}
Differentiating the implicit equation $(\kappa^*p_2)(x,f(x)) = 0$, we obtain
\begin{equation*}
    f'(x) = -\frac{(\partial_x (\kappa^*p_2))(x,f(x))}{(\partial_\xi (\kappa^*p_2))(x,f(x))}.
\end{equation*}
Here we note that the denominator is a smooth function that does not vanish around $x=0$ by Step 1. 
For the higher order derivatives, we inductively obtain
\begin{equation}\label{Eq:highderofinproof}
    f^{(k)}(x)=-\frac{(\partial_x^k(\kappa^*p_2))(x,f(x))}{(\partial_\xi(\kappa^*p_2))(x,f(x))}+g_k(x),
\end{equation}
where $g_k$ is a sum of terms having a factor of $(\partial_x^l(\kappa^*p_2))(x,f(x))$ with $l\in\{0,1,\ldots,k-1\}$ and hence $g_k(0)=0$ for $k\in \{1,\hdots,m\}$ by \eqref{Eq:tildep_2}.
In fact, this formula is true for $k=1$ so that
\begin{equation*}
    \frac{\partial}{\partial x}\left((\partial_x^k(\kappa^*p_2))(x,f(x))\right)
    =(\partial_x^{k+1}(\kappa^*p_2))(x,f(x))
    -\frac{(\partial_x(\kappa^*p_2))(x,f(x))(\partial_\xi\partial_x^k(\kappa^*p_2))(x,f(x))}{(\partial_\xi(\kappa^*p_2))(x,f(x))},
\end{equation*}
which proves \eqref{Eq:highderofinproof}.
Now the identity \eqref{Eq:highderofinproof} with \eqref{Eq:tildep_2} and $g_k(0)=0$ for $k\in \{1,\hdots,m\}$ shows \eqref{Eq:Derivative-f}.

\underline{Step 4}: Finally,  we see that $c=(\partial_{\xi}(\kappa^*p_2)(0,0))^{-1}$ satisfies the property in Theorem \ref{Thm:Matrix-NF} $(v)$. Suppose $m\geq 2$. In this case, one has $H_{p_1}p_2(0,0)=\{p_1,p_2\}(0,0)=0$ and hence $p_2(x,\xi)=\tilde{c}p_1(x,\xi)+\BigO{|(x,\xi)|^2}$ near $(x,\xi)=(0,0)$ with a constant $\tilde{c}\in \Rr$. The condition $\partial_{(x,\xi)} p_2(0,0)\neq0$ implies that 
\begin{equation*}
    \tilde{c}=\frac{\braket{\partial_{(x,\xi)} p_1(0,0),\partial_{(x,\xi)} p_2(0,0)}}{|\partial_{(x,\xi)} p_1(0,0)|^{2}}\neq0.
\end{equation*}
Then, $\kappa^*p_2(x,\xi)=\tilde{c}\kappa^*p_1(x,\xi)+\BigO{|(x,\xi)|^2}
=\tilde{c}\xi+\BigO{|(x,\xi)|^2}$ gives $c^{-1}=\partial_{\xi}(\kappa^*p_2)(0,0)=\tilde{c}$. 
\end{proof}

\begin{lemma}\label{Lemma:QuantumscalarNF}
Let $\kappa$ be a symplectomorphism as in Lemma~\ref{Lemma:ClassicalNF}. 
Then there exist an operator $F_0$ that quantizes $\kappa$ and is microlocally invertible near $((0,0),(0,0))$ in the sense of Definition \ref{def:FIO} $(i)$ and three pseudodifferential operators $A=a^w(x,hD_x)$, $B_1=b_1^w(x,hD_x)$, $B_2=b_2^w(x,hD_x)$ with $a,b_1,b_2\in C_b^\infty(\Rr^2)$ such that 
\begin{equation*}
    (F_0B_1)^{-1}P_1(F_0B_1)\equiv hD_x,\quad
    (F_0B_2)^{-1}P_2(F_0B_2)\equiv A(hD_x-f(x)),
\end{equation*}
microlocally near $((0,0),(0,0))$, $a(0,0)=a_0(0,0)+\BigO{h}$, and $b_j(0,0)=1$ ($j=1,2$).
In particular, $A$ is microlocally invertible near $((0,0),(0,0))$.
\end{lemma}

\begin{proof}
By Lemma \ref{Lemma:ClassicalNF} and \cite[Theorem 11.6]{Zwo}, there exists a unitary operator $F_0$ that quantizes $\kappa$ microlocally near $((0,0),(0,0))$. In particular, we have
\begin{align*}
F_0^{-1}P_1F_0\equiv hD_x+h\varepsilon_1^w,\quad F_0^{-1}P_2F_0\equiv A_0(hD_x-f(x))+h\varepsilon_2^w,\,\, \stext{m.l. near} ((0,0),(0,0))
\end{align*}
with symbols $\varepsilon_1,\varepsilon_2\in C_b^\infty$, where we put $A_0=a_0^w(x,hD_x)$ with $a_0$ defined in Lemma~\ref{Lemma:ClassicalNF}.
The proof of \cite[Theorem 12.3]{Zwo} (see \cite[$(12.2.11)$]{Zwo}) shows the existence of pseudodifferential operators $B_1=b_1^w$ and $B_2=b_2^w$ such that $b_1(0,0)=b_2(0,0)=1$ and
\begin{align*}
B_1^{-1}(hD_x+h\varepsilon_1^w)B_1\equiv hD_x,\quad B_2^{-1}\left(hD_x-f(x)+hA_0^{-1}\varepsilon_2^w\right)B_2\equiv hD_x-f(x)
\end{align*}
microlocally near $ ((0,0),(0,0))$.
We used the fact that $A_0$, $B_1$ and $B_2$ are microlocally invertible near $((0,0),(0,0))$ thanks to Proposition \ref{prop:elliptic}. Now the lemma is proven with $A=B_2^{-1}A_0B_2$. 
\end{proof}

Now we prove Theorem~\ref{Thm:Matrix-NF}.
\begin{proof}[Proof of Theorem~\ref{Thm:Matrix-NF}]
We define $\mc{F}:=\diag(F_1\ F_2)$ with $F_1=F_0B_1$ and $F_2=F_0B_2$. Since $F_0$ quantizes $\kappa$, the operators $F_1$ and $F_2$ quantize $\kappa$ as well by the composition rule of the psedudodifferential operators (Proposition~\ref{Prop:FIO-conjugate} $(ii)$). This proves $(ii)$. Lemma~\ref{Lemma:QuantumscalarNF} yields
\begin{align*}
\mc{F}^{-1}\mathscr{P}\mc{F}\equiv \begin{pmatrix}
hD_x&hF_1^{-1}Q_1F_2\\
hF_2^{-1}Q_2F_1&A(hD_x-f(x))
\end{pmatrix}\equiv \begin{pmatrix}
1&0\\
0&A
\end{pmatrix} \begin{pmatrix}
hD_x&hR_1\\
hR_2&hD_x-f(x)
\end{pmatrix}
\end{align*}
m.l. near $((0,0),(0,0))$, where we set $R_1=F_1^{-1}Q_1F_2$ and $R_2=A^{-1}F_2^{-1}Q_2F_1$. Thus, $(i)$ follows.

Since $F_0$ quantizes $\kappa$, its conjugation $F_0^{-1}Q_jF_0$ is a pseudodifferential operator with symbol $\kappa^*q_j+\BigOO{h}{C_b^\infty}$ near $(0,0)$ (see Proposition~\ref{Prop:FIO-conjugate}). 
Recall that we have put $F_j=F_0B_j$ with pseudodifferential operators $B_j=b_j^w$ with $b_j(0,0)=1$.
Therefore, $R_1$ and $R_2$ are pseudodifferential operators whose symbols $r_1,r_2\in C_b^\infty$ satisfy $r_1=b_1^{-1}b_2\kappa^{*}q_1+\BigOO{h}{C_b^\infty}$ and $r_2=a^{-1}b_1b_2^{-1}\kappa^*q_2+\BigOO{h}{C_b^\infty}$ near $(0,0)$ by the composition rule. Next, we observe that $a(0,0)=a_3(0,0)+\BigO{h}=\partial_{\xi}(\kappa^*p_2)(0,0)(\neq 0)+\BigO{h}$. 
Hence $r_2(0,0)=a(0,0)^{-1}q_2(0,0)=cq_2(0,0)+\BigO{h}$, where we recall $\kappa(0,0)=(0,0)$ and that we have defined $c=(\partial_{\xi}(\kappa^*p_2)(0,0))^{-1}$. 
The relation $r_1(0,0)=q_1(0,0)+\BigO{h}$ is proved similarly. 
This concludes the proof of the first statement of $(iv)$.

The statement $(v)$ and the second statement of $(iv)$ follow from Lemma \ref{Lemma:ClassicalNF}. Moreover, we may assume that $F_1$ and $F_2$ are microlocally unitary at $((0,0),(0,0))$ thanks to Proposition \ref{Prop:FIO-conjugate} $(ii)$. This with the microlocal invertibility of $A$ (Lemma \ref{Lemma:QuantumscalarNF}) proves $(iii)$,
which concludes the proof.
\end{proof}

\section{Microlocal connection formula for the reduced equation}\label{S:CFReduced}

In this section, we work with the reduced operator
\begin{align*}
P=P_{\rm nf}=\begin{pmatrix}
hD_x&hR_1\\
hR_2&hD_x-f(x)
\end{pmatrix}
\end{align*}
where $R_1,R_2\in \Psi_h^{\mathrm{comp}}$
and
$f\in C_b^\infty(\Rr)$ is a real-valued function such that
\begin{align}\label{eq:fassumption}
f^{(j)}(0)=0\quad \text{for}\quad j=0,1,\hdots,m-1,\quad f^{(m)}(0)\neq 0.
\end{align}
Note that one has $\gamma_1^\inc=\{\xi=0,\,x<0\}\cap\Omega$, $\gamma_1^\out=\{\xi=0,\,x>0\}\cap\Omega$, $\gamma_2^\inc=\{\xi=f(x),\,x<0\}\cap\Omega$, and $\gamma_2^\out=\{\xi=f(x),\,x>0\}\cap\Omega$.
We fix $I^\inc\Subset(-\infty,0)$, $I^\out\Subset(0,+\infty)$ and two microlocal solutions $\tw_j^\dir$ of $Pu = 0$ near each $\rho_j^\dir\in\gamma_j^\dir$ for $\dir = \inc,\out$ and $j = 1,2$ such that locally for $x \in I^\dir$,
\begin{align}\label{eq:normalWKB}
    \tw_1^\dir(x) =
    \begin{pmatrix}
    1 + \BigOO{h}{C^\infty_b}\\ \BigOO{h}{C^\infty_b}
    \end{pmatrix}
    \stext{and}
    \tw_2^\dir(x) = e^{\frac{i}{h}F(x)}
    \begin{pmatrix}
    \BigOO{h}{C^\infty_b}\\ 1 + \BigOO{h}{C^\infty_b}
    \end{pmatrix}.
\end{align}
Such solutions exist as stated in Proposition \ref{prop:modelWKBconst}. For $R,R'\in \Psi_h^{\mathrm{comp}}$ and fixed $x_0<0$, we define
\begin{align*}
&F(x):=\int_0^xf(y)dy,\quad \Gamma_{0,R}v(x):=\int_{x_0}^xRv(y)dy,\quad \Gamma_{\pm,R}v(x):=\int_{x_0}^{x}e^{\pm\frac ih F(y)}Rv(y)dy.
\end{align*}
Moreover, we define
\begin{align}\label{eq:KRdef}
\Gamma_+:=\Gamma_{+,(R_1)_f},\quad \Gamma_-=\Gamma_{-,R_2},\quad K_{1}:=-\Gamma_+\circ\Gamma_-,\quad K_2:=-\Gamma_+\circ\Gamma_-,
\end{align}
where we set $R_f:=e^{-\frac{i}{h}F(x)}Re^{\frac{i}{h}F(x)}\in \Psi_h^{\mathrm{comp}}$.
From \eqref{eq:fassumption}, it is easy to see
\begin{align}\label{eq:Fcond}
F^{(j)}(0)=0\quad \text{for}\quad j=0,1,\hdots,m,\quad F^{(m+1)}(0)=f^{(m)}(0)\neq 0.
\end{align}

We consider the following microlocal Cauchy problem:
\begin{align}\label{eq:modelCauchy}
Pu\equiv 0\quad\text{m.l. on $\Omega$},\quad u\equiv \alpha_j^{\inc}\tw_j^{\inc} \quad\text{m.l. near $\rho_j^{\inc}$},
\end{align}
where $\Omega\Subset \Rr^2$ is an open neighborhood of $(0,0)$, $\alpha_j^{\inc}\in \mathbb{C}$, and $\rho_j^{\inc}\in \gamma_j^{\inc}\cap \Omega$ ($j=1,2$).

The next propositions are the main results of this section. The first one states the uniqueness and the existence of the microlocal Cauchy problem \eqref{eq:modelCauchy} and the second one provides the asymptotic expansion of the transfer matrix.

\begin{prop}[Existence and uniqueness for the reduced problem]\label{Prop:Uniqueness-NF}

There exist $h_0\in (0,1]$ and an open neighborhood $\Omega\Subset \Rr^2$ of $(0,0)$ such that the following holds for $h\in(0, h_0]$: 

\noindent$(i)$ The solution of the microlocal Cauchy problem \eqref{eq:modelCauchy} is unique. Namely, if a tempered family $u$ satisfies \eqref{eq:modelCauchy} with $\alpha_1^\inc=\alpha_2^\inc=0$ for some $\rho_j^{\inc}\in \gamma_j^{\inc}\cap \Omega$ $(j=1,2)$, then
\begin{equation*}
u\equiv0\quad\text{m.l. on }\Omega.
\end{equation*}

\noindent$(ii)$ Let $\rho_j^{\inc}\in \gamma_j^{\inc}\cap \Omega$. Then, for each $(\alpha_{1}^{\inc},\alpha_2^{\inc})\in\mathbb{C}^2$, there exists a tempered family $u={}^t(u_1,u_2)$ such that \eqref{eq:modelCauchy} is true. Moreover, if $\alpha_j^{\inc}$ is depending on $h>0$ and $|\alpha_j^{\inc}|=\BigO{1}$ as $h\to 0$, then
\begin{align}
u_1(x) \equiv & \ \alpha_1^{\inc}-i\alpha_2^{\inc}\Gamma_{+}(\mathbbm{1})(x)+\BigOO{h^{\frac{2}{m+1}}\logdelt }{L^{\infty}}\label{eq:exasymexu_1}\\
u_2(x) \equiv & \ e^{\frac{i}{h}F(x)}\left(\alpha_2^{\inc}-i\alpha_1^{\inc}\Gamma_{-}(\mathbbm{1})(x) \right)+ \BigOO{h^{\frac{2}{m+1}}\logdelt}{L^{\infty}}\label{eq:exasymexu_2}
\end{align}
hold m.l. on $\Omega$. Here, $\mathbbm{1}$ denotes the constant function equal to one. 

\end{prop}

Now we can define the transfer matrix $T$ with respect to the basis which we have fixed in \eqref{eq:normalWKB}. In fact, from Proposition \ref{Prop:Uniqueness-NF}, we can define a linear map $(\alpha_1^\inc,\alpha_2^\inc) \in \Cc^2 \mapsto u$ whose range is included in the set of tempered families $u$ satisfying \eqref{eq:modelCauchy}. 
From Proposition \ref{Prop:matrixmicsol}, for $j=1,2$ there exists $\alpha_j^{\out}\in\mathbb{C}$ such that $u=\alpha_j^{\out}\tw_j^{\out}$ m.l. 
near
$\rho_j^{\out}$. 
Then we have a linear map $u \mapsto (\alpha_1^\out,\alpha_2^\out) \in \Cc^2$. We let $T:(\alpha_1^\inc,\alpha_2^\inc)\mapsto (\alpha_1^\out,\alpha_2^\out)$ be the composition of these two linear maps.

\begin{prop}[Microlocal connection formula for the reduced operator] \label{Prop:AsymptoticsOfTransferMatrix}
The matrix $T$ admits the asymptotic formula 
\begin{equation}\label{Eq:MicrolocalTransferMatrix}
T = 
\begin{pmatrix}
    1& 0\\
    0 & 1\\
\end{pmatrix}
-i h^{\frac1{m+1}}
\begin{pmatrix}
    0& r_1(0,0)\omega\\
    r_2(0,0)\overline{\omega} & 0\\
\end{pmatrix}
+ \BigO{h^{\frac{2}{m+1}}\logdelt},
\end{equation}
where $r_j$ are symbols of $R_j$: $R_j=r_j^w(x,hD_x)$ and
\begin{align}\label{Eq:omega}
\omega=2\mu_m\left(\frac{\sgn(f^{(m)}(0))}{2(m+1)}\pi\right)
\mathbf{\Gamma}\left(\frac{m+2}{m+1}\right) 
\left(\frac{(m+1)!}{|f^{(m)}(0)|}\right)^{\frac{1}{m+1}}. 
\end{align}
with $\mu_m(\theta) := 2^{-1}(e^{i\theta} + e^{i(-1)^{m+1}\theta})$. 
\end{prop}

We prove those two propositions in the subsections \ref{SubS:uniexist} and \ref{SubS:ProofConnectionFormula} after a few preliminary lemmas.
By Lemma \ref{lem:Psicomp}, we may assume that for a compact neighborhood $S\subset \Rr$ of $0$,
\begin{align}\label{eq:suppcondR}
\mathrm{supp}~ E_{R_1}\cup \mathrm{supp}~E_{R_2}\subset S\times S,\quad f'(x)\neq 0\quad \text{for}\quad x\in S\setminus \{0\},
\end{align}
where $E_{R_j}$ denotes the integral kernel of $R_j$.

\subsection{Microlocal Duhamel principle of the scalar equation}\label{SubS:Duhamel}


\begin{lemma}[Microlocal Duhamel formula for the scalar equation]\label{Lem:Duhamel}
Let $u,v$ be tempered families of $L^2$ functions, $R\in \Psi_h^{\mathrm{comp}}$, and $\Omega\subset \Rr^2$ be an open neighborhood of $(0,0)$ such that its intersection with $\{\xi=f(x)\}$ is connected.
Suppose that $u$ and $v$ 
satisfy 
\begin{equation}\label{eq:Single-eq0}
(hD_x-f(x))u+hRv\equiv0\quad \text{m.l. on $\Omega$}.
\end{equation}
Then, for each $x_0\in \Rr$, there exists a constant $C\in\Cc$ such that 
\begin{equation}\label{eq:Duh0}
    u(x)\equiv e^{\frac ihF(x)}\left(C-i\int_{x_0}^xe^{-\frac{i}{h}F(y)}Rv(y)dy\right) \quad \text{m.l. on $\Omega$}.
\end{equation}
\end{lemma}

\begin{proof}

Put $w(x):=-ie^{\frac{i}{h}F(x)}\int_{x_0}^xe^{-\frac{i}{h}F(y)}Rv(y)dy$. By \eqref{eq:Single-eq0}, one has
\begin{align*}
(hD_x-f(x))(u(x)-w(x))\equiv 0\quad \text{m.l. on $\Omega$}.
\end{align*}
On the other hand, we clearly have $(hD_x-f(x))(e^{\frac{i}{h}F(x)})=0$.
By the uniqueness of the microlocal solutions (Proposition \ref{prop:scalarmicsol} $(i)$), they are linearly dependent. This proves \eqref{eq:Duh0}.

\end{proof}

\subsection{Estimates for integral operators}

The purpose of this subsection is to derive some estimates for the integral operators related to the Duhamel formula \eqref{eq:Single-eq0}. 
The following proposition corresponds to \cite[Proposition 3.1]{FMW1} or \cite[Proposition 3.8]{AFH}.

\begin{prop}\label{prop:normbdd}

Fix $i=1,2$.

\noindent$(i)$
There exists an $h$-independent constant $C>0$ such that   
\begin{align}
 \|\Gamma_{\pm}\|_{L^{\infty}\to L^{\infty}}\le C,\quad \|\Gamma_{\pm}(\mathbbm{1})\|_{L^\infty}\leq Ch^{\frac{1}{m+1}},\quad \label{Eq:Esti-Gamma} \\
\|K_{j}\|_{L^\infty\to L^\infty}\le Ch^{\frac1{m+1}},\quad \|K_{j}(\mathbbm{1})\|_{L^\infty}\leq Ch^{\frac{2}{m+1}}\logdelt\label{Eq:Esti-Kk1}
\end{align}
for $0<h\leq 1$ and $j=1,2$. 

\noindent$(ii)$ Let $I_{\inc}\Subset (-\infty,0)$ and $I_{\out}\Subset (0,\infty)$. Then,
\begin{align*}
&\|\Gamma_{\pm}(\mathbbm{1})\|_{L^{\infty}(I_{\inc})}=\BigO{h},\quad \|K_{i}(\mathbbm{1})\|_{L^{\infty}(I_{\inc})}=\BigO{h}\\ 
&\|\Gamma_{\pm}(\mathbbm{1})-\omega_{\pm} r_{\pm}(0,0)h^{\frac{1}{m+1}} \|_{L^{\infty}(I_{\out})}=\BigO{h^{\frac{2}{m+1}}}
\end{align*}
as $h\to 0$, where 
$\omega_+=\omega$, $\omega_-=\overline{\omega}$, $r_{+}=r_1$, and $r_-=r_2$. We recall that $\omega$ is given in \eqref{Eq:omega}.

\end{prop}

\begin{rmk}
In general, $\|\Gamma_{j}u\|_{L^\infty}=\BigO{h^{\frac{1}{m+1}}}$ and $\|K_{j}u\|_{L^\infty}\leq Ch^{\frac{2}{m+1}}\logdelt$ do not hold as $h\to 0$ even if $u=u_h$ is a uniformly bounded family in $L^{\infty}$ and its optimal bound of the first estimate is $\BigO{1}$. The second inequalities in \eqref{Eq:Esti-Gamma} and \eqref{Eq:Esti-Kk1} mean that these bounds can be improved if $u$ is the constant function $\mathbbm{1}$.
\end{rmk}


\begin{proof}
First, we observe that the conditions  \eqref{eq:Fcond} and \eqref{eq:suppcondR} allow us to use the degenerate stationary phase theorem (Lemma \ref{Lemma:degst}) in the following proof.

\noindent$(i)$
First, we claim that there exists $C>0$ such that 
\begin{align}\label{eq:Gammabound}
\|\Gamma_{\pm}\|_{L^{\infty}\to W^{1,\infty}}\leq C,\quad \left\|\Gamma_{\pm}\right\|_{W^{1,\infty}\to L^\infty}
    \le Ch^{\frac 1{m+1}}
\end{align}
for $0<h\leq 1$, where $\|u\|_{W^{1,\infty}}=\|u\|_{L^{\infty}}+\|D_xu\|_{L^{\infty}}$. We prove it for $\Gamma_+$ and the proof for $\Gamma_-$ is similar.
By the definition of $\Psi_h^{\mathrm{comp}}$, there exists a compact interval $J\subset \Rr$ (independent of $0<h\leq 1$) such that $(R_1)_f=\chi_J(R_1)_f$, where $\chi_J$ is the characteristic function of the subset $J$. We also note that $(R_1)_f\in \Psi_h^{\mathrm{comp}}$ by Proposition \ref{Prop:conjugatePDO} and $(R_1)_f=\chi_J(R_1)_f$.
By H\"older's inequality and the definition of $\Gamma_+$, 
\begin{align*}
\|\Gamma_{+}u\|_{L^{\infty}}\leq& \|\chi_J\|_{L^1}\|(R_1)_f\|_{L^{\infty}\to L^{\infty}}\|u\|_{L^{\infty}}\leq C\|u\|_{L^{\infty}},
\end{align*}
where we use Lemma~\ref{Lem:L-infty-bdd} for the uniform boundedness of $(R_1)_f$ on $L^{\infty}$. Moreover, 
\begin{align*}
D_x\circ \Gamma_{+}u(x)=-ie^{\frac ihF(x)}\chi_J(x)(R_1)_fu(x),
\end{align*}
which implies $\|D_x\circ \Gamma_{+}\|_{L^{\infty}\to L^{\infty}}\leq \|\chi_J\|_{L^{\infty}}\|(R_1)_f\|_{L^{\infty}\to L^{\infty}}\leq C$. These prove the first inequality in \eqref{eq:Gammabound}. The second inequality in \eqref{eq:Gammabound} follows 
from the degenerate stationary phase estimate \eqref{Eq:degstEstimate}:
\begin{align*}
\left|\int_{x_0}^xe^{\frac{i}{h}F(x)}\chi_J(y)(R_1)_fu(y)dy \right|\leq Ch^{\frac{1}{m+1}}\|\chi_J(R_1)_fu\|_{W^{1,\infty}}\leq C'\|u\|_{W^{1,\infty}},
\end{align*}
where we use the trivial bound $\sup_{0<h\leq 1}h^{\frac{1}{m+1}}\log(1/h)<\infty$ and the uniform boundedness of $(R_1)_f=\chi_J(R_1)_f$ between $W^{1,\infty}$ stated in Lemma \ref{Lem:L-infty-bdd}.

The first estimate in \eqref{Eq:Esti-Gamma} directly follows from \eqref{eq:Gammabound}.
Using \eqref{eq:Gammabound} again, we obtain
\begin{align*}
    \|K_{1}\|_{L^{\infty}\to L^{\infty}}\leq \|\Gamma_{1}\|_{W^{1,\infty}\to L^\infty} \|\Gamma_{2}\|_{L^{\infty}\to W^{1,\infty}}\leq C'h^{\frac{1}{m+1}},
\end{align*}
which shows the first bound in \eqref{Eq:Esti-Kk1} regrading $K_{1}$. The proof for $K_{2}$ is similar.

Next, we deal with the second inequality in \eqref{Eq:Esti-Gamma}.
We note 
\begin{align}\label{Eq:Gammapexp}
\Gamma_{+}(\mathbbm{1})(x)=\int_{x_0}^xe^{+\frac{i}{h}F(y)}a_h(y)dy,
\end{align}
where $a_h(x):=\chi_J(x)(R_1)_f(\mathbbm{1})(x)$. Clearly, for each $\alpha\in\mathbb{N}$,  $|\partial_x^{\alpha}a_h(x)|$ are uniformly bounded with respect to $0<h\leq 1$ and $x\in \Rr$ by \eqref{eq:PseuDO-WKB}. Then, the inequality $\|\Gamma_{+}(\mathbbm{1})\|_{L^\infty}\leq Ch^{\frac{1}{m+1}}$ 
follows from the degenerate stationary phase estimate \eqref{Eq:degstEstimate}.
The proof for $\Gamma_-$ is similar.

Finally, we prove the second estimate in \eqref{Eq:Esti-Kk1}. We deal with $K_{1}$ only. Set $g(x)=(R_1)_f\Gamma_{-}(\mathbbm{1})(x)$. Since $(R_1)_f\in \Psi_h^{\mathrm{comp}}$, the function $w$ is supported in a compact subset which is independent of $0<h\leq 1$.
By the degenerate stationary phase estimate \eqref{Eq:degstEstimate}, we have
\begin{align}\label{Eq:K1Bound}
|K_{1}(\mathbbm{1})(x)|=\left|\int_{x_0}^xe^{\frac{i}{h}F(y)}g(y)  dy\right|\leq Ch^{\frac{1}{m+1}}\|g\|_{L^{\infty}}+Ch^{\frac{2}{m+1}}\logdelt~\|D_xg\|_{L^{\infty}},
\end{align}
where $C>0$ is independent of $x\in \Rr$ and $0<h\leq 1$. 
On the other hand, $\|g\|_{L^{\infty}}=\BigO{h^{\frac{1}{m+1}}}$ and $\|D_xg\|_{L^{\infty}}=\BigO{1}$ due to \eqref{Eq:Esti-Gamma}, \eqref{eq:Gammabound}, and Lemma \ref{Lem:L-infty-bdd}. This completes the proof.

\noindent$(ii)$ 
The derivative of the phase $F$ of $e^{\frac{i}{h}F(x)}$ does not vanish on $S$ appearing in \eqref{eq:suppcondR}, so 
an
integration by parts and the $W^{1,\infty}$-boundedness of $(R_1)_f$ and $R_2$ (Lemma \ref{Lem:L-infty-bdd}) give $\|\Gamma_{\pm}(\mathbbm{1})\|_{L^{\infty}(I_{\inc})}=\BigO{h}$ as well as $\|K_{1}(1)\|_{L^\infty(I_\inc)}=\BigO{h}$.

The last estimate comes from the degenerate stationary expansion \eqref{Eq:degstExpansion} 
due to the conditions \eqref{eq:Fcond} and \eqref{eq:suppcondR}.

\end{proof}

We need one more technical lemma.

\begin{lemma}\label{lem:inczero}
Let $I_{\inc}\Subset (-\infty,0)$ and $i=1,2$.
If $r\in L^{\infty}(\Rr)$ satisfies $\|r\|_{L^{\infty}(I_{\inc})}=\BigO{h^{\infty}}$ for all $I_{\inc}\Subset (-\infty,0)$, then 
\begin{align*}
\|\Gamma_{\pm}r\|_{L^{\infty}(I_{\inc})}=\BigO{h^{\infty}},\quad \|K_{i}r\|_{L^{\infty}(I_{\inc})}=\BigO{h^{\infty}}.
\end{align*}

\end{lemma}

\begin{proof}

We only consider $\Gamma_{+}$ and $K_{1}$. We write $\widetilde{I_\inc} := \mathrm{Conv}(I_\inc, \{x_0\})$, where $\mathrm{Conv}(A,B)$ denotes the convex hull of the sets $A,B$. Since $x_0<0$, $\|r\|_{L^{\infty}(\widetilde{I_\inc})}=\BigO{h^{\infty}}$. Moreover, 
$\|\Gamma_{+}r\|_{L^{\infty}(I_{\inc})} \leq |\widetilde{I_\inc}| \|(R_1)_fr\|_{L^{\infty}(\widetilde{I_\inc})}$ 
and 
$\|K_{1}r\|_{L^{\infty}(I_{\inc})} \leq |\widetilde{I_\inc}| \|(R_1)_f \Gamma_{-}r\|_{L^{\infty}(\widetilde{I_\inc})}$. 
Since $(R_1)_f$ and $R_2$ are pseudolocal, the estimate follows.
\end{proof}

\subsection{Existence of basis of the microlocal solution space}

In this subsection, we construct a basis of microlocal solutions near the crossing point 
via
Neumann series.

\begin{lemma}\label{lem:basisexist}
There exist tempered families $v,\tilde{v} \in L^{\infty}(\Rr;\mathbb{C}^2)$ such that the following holds as $h\to +0$: Let $\Omega\Subset \Rr^2$ be an open neighborhood of $(0,0)$ and $I_{\inc}\Subset (-\infty,0)$.

\noindent$(i)$ $\|v\|_{L^{\infty}(\Rr;\mathbb{C}^2)},\|\tilde{v}\|_{L^{\infty}(\Rr;\mathbb{C}^2)}=\BigO{1}$;

\noindent$(ii)$ The functions $v,\tilde{v}$ are exact solutions: $Pv= 0$ and $P\tilde{v}= 0$;

\noindent$(iii)$ There exists an $h$-dependent matrix $A=(A_{ij})\in M_2(\mathbb{C})$ such that $A=I+\BigO{h}$ and
\begin{align}\label{eq:uptoOhsol}
v\equiv A_{11}\tw_1^{\inc} +A_{12}\tw_2^{\inc},\quad \tilde{v}\equiv A_{21}\tw_1^{\inc} +A_{22}\tw_2^{\inc}\quad \text{m.l. on $\Omega\cap (I_{\inc}\times \Rr)$};
\end{align}

\noindent$(iv)$ We have
\begin{align*}
v(x)=&\begin{pmatrix}
1\\
-ie^{\frac{i}{h}F(x)}\Gamma_{-,R_2}(\mathbbm{1})(x) )
\end{pmatrix}+\BigOO{h^{\frac{2}{m+1}}\logdelt}{L^{\infty}(\Rr;\mathbb{C}^2)},\\
\tilde{v}(x)=&\begin{pmatrix}
-i\Gamma_{+,(R_1)_f}(\mathbbm{1})(x) \\
e^{\frac{i}{h}F(x)}
\end{pmatrix}+\BigOO{h^{\frac{2}{m+1}}\logdelt}{L^{\infty}(\Rr;\mathbb{C}^2)}.
\end{align*}
\end{lemma}

\begin{proof}
We observe that $\|K_
j\|_{L^{\infty}}=\BigO{h^{\frac{1}{m+1}}}$ by Proposition \ref{prop:normbdd} and hence the Neumann series $(I-K_j)^{-1}=\sum_{\ell=0}^{\infty}K_j^{\ell}$ converges in the operator topology in $L^{\infty}(\Rr)$ for sufficiently small $h>0$.

We observe that the exact equation $Pu=0$ for $u={}^t(u_1,u_2)$ can be written as in \eqref{eq:intromodelDuh2} (with the notation \eqref{eq:KRdef}).
Let us define $v={}^t(v_1,v_2)$ and $\tilde{v}={}^t(\tilde{v}_1,\tilde{v}_2)$ as \eqref{eq:intromodelDuh2} with $(\alpha_1^{\inc},\alpha_2^{\inc})=(1,0)$ and $(\alpha_1^{\inc},\alpha_2^{\inc})=(0,1)$ respectively:
\begin{align*}
v_1:=&\sum_{\ell=0}^{\infty}K_1^{\ell}(\mathbbm{1}),\,\, v_{2}:=-ie^{\frac{i}{h}F(x)}\Gamma_{-}(v_1),\,\,
\tilde{v}_2:=e^{\frac{i}{h}F(x)}\sum_{\ell=0}^{\infty}K_2^{\ell}(\mathbbm{1}),\,\, \tilde{v}_{1}:=-i\Gamma_{+}(\tilde{v}_2).
\end{align*}
Then it is easy to check the property $(i)$ 
because
$\|K_
j\|_{L^{\infty}}=\BigO{h^{\frac{1}{m+1}}}$. The property $(ii)$ is clearly satisfied by the construction of $v$ and $\tilde{v}$.

We observe that $\|\sum_{\ell=1}^{\infty}K_j^{\ell}(\mathbbm{1})\|_{L^{\infty}}=\BigO{h^{\frac2{m+1}}\logdelt}$ holds for $j=1,2$ by Proposition \ref{prop:normbdd}. Then, the property $(iv)$ can be deduced from the definition of $\tilde{h}$ and $v,\tilde{v}$.

Finally, we prove $(iii)$.
By the uniqueness result of the microlocal solution away from the crossing point (Proposition \ref{Prop:matrixmicsol} $(i)$), we can find a ($h$-dependent) matrix $A=(A_{ij})\in M_2(\mathbb{C})$ such that \eqref{eq:uptoOhsol} holds. Thus, we only need to prove $A=I+O(h)$ as $h\to 0$. 
We observe that 
\begin{align*}
\|\Gamma_{\pm}(\mathbbm{1})\|_{L^{\infty}(I_{\inc})}=\BigO{h},\quad  \|K_j^\ell(\mathbbm{1})\|_{L^{\infty}(I_{\inc})}=\BigO{h^j}
\end{align*}
for $j=1,2$ and $\ell\geq 1$ by Proposition \ref{prop:normbdd} $(ii)$. Using these estimates and the identity \eqref{eq:normalWKB}, we obtain
\begin{align*}
v(x)&=\begin{pmatrix}
1\\
0
\end{pmatrix}+\BigOO{h}{L^{\infty}(I_{\inc};\mathbb{C}^2)}=\tw_1^{\inc}(x)+\BigOO{h}{L^{\infty}(I_{\inc};\mathbb{C}^2)},\\
\tilde{v}(x)&=e^{\frac{i}{h}F(x)}\begin{pmatrix}
0\\
1
\end{pmatrix}+\BigOO{h}{L^{\infty}(I_{\inc};\mathbb{C}^2)}=\tw_2^{\inc}(x)+\BigOO{h}{L^{\infty}(I_{\inc};\mathbb{C}^2)}.
\end{align*}
uniformly in $x\in I_{\inc}$.
Therefore, the identity $A=I+O(h)$ follows.

\end{proof}

\subsection{Proof of the existence and the uniqueness of the microlocal solution}\label{SubS:uniexist}

\begin{proof}[Proof of Proposition \ref{Prop:Uniqueness-NF}]

\noindent$(i)$
Let $u$ be a tempered family satisfying \eqref{eq:modelCauchy} with $\alpha_j^{\inc}=0$ for $j=1,2$.
By the uniqueness result away from the crossing point (Proposition \ref{Prop:matrixmicsol} $(i)$), it suffices to prove $u\equiv 0$ m.l. near $(0,0)$. 
Taking $\Omega$ sufficiently small and 
cutting $u$ off
appropriately in the phase space, we may assume\footnote{\eqref{eq:modelCauchy} just implies $(\mathrm{WF}_h(u_1)\cup\mathrm{WF}_h(u_2))\cap \Omega\subset \Lambda_1\cup \Lambda_2$.}
\begin{align}\label{eq:unimayassume}
\text{$u$ is compactly microlocalized and $\mathrm{WF}_h(u_1)\cup\mathrm{WF}_h(u_2)\subset \Lambda_1\cup \Lambda_2$.}
\end{align}

First, we claim 
\begin{align}\label{eq:uniqueinc0}
\|u_j\|_{L^{\infty}(I)}=\BigO{h^{\infty}}\quad \text{for all $I\Subset (-\infty,0)$ and $j=1,2$}. 
\end{align}
Using $\alpha_j^{\inc}=0$ for $j=1,2$, we have $\mathrm{WF}_h(u_j)\cap \left((-\infty,0)\times \Rr\right)=\emptyset$ by Proposition \ref{Prop:matrixmicsol} and \eqref{eq:unimayassume}.
Since $u={}^t(u_1,u_2)$ is compactly microlocalized,  we obtain \eqref{eq:uniqueinc0}.

Next, we show
\begin{align}\label{eq:uniiteid}
u_1(x)\equiv K_{1}^Nu_1(x) \quad \text{m.l. on $\Omega$}
\end{align}
for all $N\in\mathbb{N}$. By Lemma \ref{Lem:Duhamel}, there exists $C_{2,1}\in\mathbb{C}$ such that
\begin{align*}
u_2(x)\equiv e^{\frac{i}{h}F(x)}\left(C_{2,1}-i\Gamma_{-,R_2}u_1(x) \right)\quad \text{m.l. on $\Omega$}.
\end{align*}
By Lemma \ref{lem:inczero} and \eqref{eq:uniqueinc0}, $\|u_2\|_{L^{\infty}(I)},\|\Gamma_{-}u_1\|_{L^{\infty}(I)}=\BigO{h^{\infty}}$ for all $I\Subset (-\infty,0)$.
These estimates imply $C_{2,1}=O(h^{\infty})$ and hence
\begin{align*}
u_2(x)\equiv -ie^{\frac{i}{h}F(x)}\Gamma_{-}u_1(x)\quad \text{m.l. on $\Omega$}.
\end{align*}
Substituting this identity to the second equation in \eqref{eq:modelCauchy}, we have
\begin{align*}
hD_xu_1-ihR_1\left(e^{\frac{i}{h}F(x)}\Gamma_{-}u_1(x) \right)\equiv 0\quad \text{m.l. on $\Omega$}.
\end{align*}
By using Lemma \ref{Lem:Duhamel} (with $f=0$) again, we find $C_{1,1}\in\mathbb{C}$ satisfying
\begin{align*}
u_1(x)\equiv C_{1,1}+K_{1}u_1 \quad \text{m.l. on $\Omega$},
\end{align*}
where we use $R_1e^{\frac{i}{h}F(x)}=e^{\frac{i}{h}F(x)}(R_1)_f$.
A similar argument shows $C_{1,1}=O(h^{\infty})$ and hence
\begin{align*}
u_1(x)\equiv K_{1}u_1(x) \quad \text{m.l. on $\Omega$}.
\end{align*}
This proves \eqref{eq:uniiteid} with $N=1$.
Continuing this procedure, we obtain\footnote{The identity $v_1\equiv v_2$ m.l. on $\Omega$ does not imply $K_{1}v_1\equiv K_{1}v_2$ m.l. on $\Omega$ in general, since $K_{1}$ is not pseudolocal. Therefore, we need to repeat the procedure used in the proof for $N=1$.} \eqref{eq:uniiteid} for all $N\in\mathbb{N}$.

Since $u$ is tempered, $\|u_j\|_{L^{\infty}}=\BigO{h^{-M}}$ as $h\to 0$ for some $M>0$.
By Lemma \ref{prop:normbdd}, we have $ \|K_{1}^Nu_1\|_{L^{\infty}}=\BigO{h^{\frac{N}{m+1}-M}}$ for all $N\in\mathbb{N}$ and hence
\begin{align*}
u_1\equiv \BigOO{h^{\frac{N}{m+1}-M}}{L^{\infty}} \quad \text{m.l. on $\Omega$}
\end{align*}
by \eqref{eq:uniiteid}. This shows $u_1\equiv 0$ m.l. on $\Omega$. The proof for $u_2\equiv 0$ m.l. on $\Omega $ is similar.

\vspace{3mm}

\noindent$(ii)$ 
We use the symbols $v,\tilde{v},A$ as in Lemma \ref{lem:basisexist}. Since $A=I+\BigO{h}$ as $h\to 0$, $A$ is invertible when $h$ is sufficiently small. Then $B:=A^{-1}=(B_{ij})$ that exists and $B=I+\BigO{h}$ holds for such $h>0$. By the definition of $B$ and Lemma \ref{lem:basisexist} $(iii)$, we have
\begin{align*}
\tw_1^{\inc}\equiv B_{11}v+B_{12}\tilde{v},\quad \tw_2^{\inc}\equiv B_{21}v+B_{22}\tilde{v}\quad \text{m.l. on $\Omega\cap( I_{\inc}\cap \Rr$)}.
\end{align*}

Now we define
\begin{align*}
u:=\alpha_1^{\inc}(B_{11}v+B_{12}\tilde{v})+\alpha_2^{\inc}(B_{21}v+B_{22}\tilde{v}).
\end{align*}
It is easy to see $u\equiv \alpha_1^{\inc}\tw_1^{\inc}+\alpha_2^{\inc}\tw_2^{\inc}$ m.l. on $\Omega \cap (I_{\inc}\times \Rr)$ and hence
\begin{align*}
u\equiv \alpha_{j}^{\inc}\tw_j^{\inc}\quad \text{m.l. on $\rho_j^{\inc}$}
\end{align*}
since $\tw_1^{\inc}\equiv 0$ and $\tw_2^{\inc}\equiv 0$ m.l. on $\rho_2^{\inc}$ and $\rho_1^{\inc}$ respectively due to Proposition \ref{Prop:matrixmicsol} $(ii)$.
Since $Pv = 0$ and $P\tilde{v} = 0$,
we have $Pu\equiv 0$ m.l. on $\Omega$. 
The asymptotic expansions \eqref{eq:exasymexu_1} and \eqref{eq:exasymexu_2} directly follows from Lemma \ref{lem:basisexist} $(iv)$.
This completes the proof.   
\end{proof}

\subsection{Proof of the connection formula}\label{SubS:ProofConnectionFormula}

\begin{proof}[Proof of Proposition \ref{Prop:AsymptoticsOfTransferMatrix}]

We write $T=(t_{ij})_{(i,j)}$. We calculate the asymptotic formula of the coefficients $t_{11}$ and $t_{21}$. The 
computation
of $t_{12},t_{22}$ is similar \footnote{To do this, one has to calculate the symbol of $(R_1)_f$ at $(0,0)$. Proposition \ref{Prop:conjugatePDO} shows that the symbol of $(R_1)_f$ is $r_1(x,x+f(x))$ modulo $O(h)$ term. Since $f(0)=0$, its value at $(0,0)$ is $r_1(0,0)+O(h)$. Then a similar calculation to the one below leads to the asymptotic expansion of $t_{12},t_{22}$.}. By Proposition \ref{Prop:Uniqueness-NF} $(ii)$, we can find a microlocal solution $u={}^t(u_1,u_2)$ of $Pu\equiv 0$ near $(0,0)$ such that for $\dir\in \{\inc,\out\}$ and $j=1,2$, 
\begin{align}\label{Eq:connforass}
u\equiv \alpha_j^{\dir}\tw_j^{\dir}\quad \text{m.l. on $\gamma_j^{\dir}\cap \Omega$}\quad\quad \text{with}\quad (\alpha_1^{\inc},\alpha_2^{\inc})=(1,0).
\end{align}
and $\|u_j\|_{L^{\infty}}=\BigO{1}$ as $h\to 0$ for $j=1,2$. 
By the definition of the transfer matrix $T$, we have $\alpha_1^{\out}=t_{11}$ and $\alpha_2^{\out}=t_{21}$. Thus, our task is to calculate the asymptotic formula for $\alpha_j^{\out}$ for $j=1,2$.

By Proposition \ref{Prop:Uniqueness-NF} $(ii)$ with $(\alpha_1^{\inc},\alpha_2^{\inc})=(1,0)$, we have
\begin{align*}
\begin{pmatrix}
u_1(x)\\
u_2(x)
\end{pmatrix}\equiv \begin{pmatrix}
1\\
-ie^{\frac{i}{h}F(x)}\Gamma_{-}(1)(x)
\end{pmatrix}+\BigOO{h^{\frac{2}{m+1}}\logdelt }{L^{\infty}}
\end{align*}
m.l. on $\Omega$. Combining this with Proposition \ref{prop:normbdd} $(ii)$, we see that for each all $I_{\out}\Subset (0,\infty)$,
\begin{align*}
\begin{pmatrix}
u_1(x)\\
u_2(x)
\end{pmatrix}\equiv& \begin{pmatrix}
1\\
-i\omega_{-} r_2(0,0)h^{\frac{1}{m+1}}e^{\frac{i}{h}F(x)}
\end{pmatrix}+\BigOO{h^{\frac{2}{m+1}}\logdelt }{L^{\infty}}\\
\equiv&f_1^{\out}(x)-i\overline{\omega} r_2(0,0)h^{\frac{1}{m+1}}f_2^{\out}(x)+\BigOO{h^{\frac{2}{m+1}}\logdelt}{L^{\infty}}
\end{align*}
m.l. on $\Omega\cap (I_{\out}\times \Rr)$, where we use $\tw_1^{\out}={}^t(1,0)+\BigOO{h}{L^{\infty}}$ and $\tw_2^{\out}={}^t(0,e^{\frac{i}{h}F(x)})+\BigOO{h }{L^{\infty}}$.
Consequently, we obtain $t_{11}=\alpha_1^{\out}=1+\BigOO{h^{\frac{2}{m+1}}\logdelt }{L^{\infty}} $ and 
$t_{21}=\alpha_2^{\out}-i\overline{\omega}  r_2(0,0)h^{\frac{1}{m+1}}+\BigOO{h^{\frac{2}{m+1}}\logdelt }{L^{\infty}} $ from \eqref{Eq:connforass}. This completes the proof.

\end{proof}

\section{Correspondence of transfer matrices}\label{S:Correspondence}
According to Theorem~\ref{Thm:Matrix-NF}, there exist an $L^2$-bounded operator $\mathcal{F}$ and an elliptic pseudodifferential operator $A$ such that
\begin{equation*}
    \mathcal{F}^{-1}\mathscr{P}\mathcal{F}\equiv\operatorname{diag}(1\ A)P_{\rm nf},\quad
    P_{\rm nf}=\begin{pmatrix}
        hD_x&hR_1\\hR_2&hD_x-f(x)
    \end{pmatrix},
\end{equation*}
microlocally near $((0,0),(0,0))$.
Let $T=T(w_1^\inc,w_2^\inc,w_1^\out,w_2^\out)$ be the transfer matrix for the original operator $\mathscr{P}$ and let $T_{\rm nf}=T_{\rm nf}( \tw_1^\inc,\tw_2^\inc,\tw_1^\out,\tw_2^\out)$ be that of the 
reduced operator $P_{\rm nf}$.
Here, $w_j^\dir$ denotes the WKB solutions for $\mathscr{P}$ given in Proposition~\ref{Prop:matrixmicsol}, and $\tw_j^\dir$ denotes that of $P_{\rm nf}$ in \eqref{eq:WKB-basis-nf} 
constructed in Proposition~\ref{prop:modelWKBconst}.
Note that the functions $\tw_2^\inc$, $\tw_2^\out$ are not normalized in the sense of Definition~\ref{def:normalized} for a notational simplicity.
One has the following relation between these matrices. 

\begin{prop}\label{prop:T-and-Tnf}
    Assume \eqref{eq:NoCaustic}.
    Then one has
    \begin{equation}
        T=I+
        \begin{pmatrix}
            0&\left|c/c'\right|^{1/2}t_{12}\\
            (\operatorname{sgn}c)\left|c/c'\right|^{-1/2}t_{21}&0
        \end{pmatrix}
        +\BigO{h^{\frac2{m+1}}\logdelt},
    \end{equation}
    where $t_{jk}$ stands for the $(j,k)$-entry of $T_{\rm nf}$, $c=(\partial_\xi(\kappa^*p_2)(0,0))^{-1}$ appearing in Theorem~\ref{Thm:Matrix-NF}, and
    \begin{equation}\label{eq:def-c'}
        c'=\frac{\left|\partial_{(x,\xi)}p_1(x,\xi)\right|}{\left|\partial_{(x,\xi)}p_2(x,\xi)\right|}.
    \end{equation}
\end{prop}

Once this proposition is proved, Theorem~\ref{Th:LSM} follows immediately from this with the asymptotic behavior of $T_{\rm nf}$ given in Proposition~\ref{Prop:AsymptoticsOfTransferMatrix}.
See Subsection~\ref{S:computation} for the detail of the computation.

\subsection{Proof of Proposition~\ref{prop:T-and-Tnf} when $c$ is positive}
In this subsection, we first prove Proposition~\ref{prop:T-and-Tnf} in the case of positive $c$.
The proof for the case of negative $c$ is provided in Subsection~\ref{SubS:negative-c}.

\begin{lemma}\label{lem:beta-with-pos-c}
    Assume that \eqref{eq:NoCaustic} holds and the constant $c$ introduced in Theorem \ref{Thm:Matrix-NF} $(iv)$ is positive. 
    Then, there exist constants $\beta_j^{\dir}\in\mathbb{C}$ such that
    \begin{equation}\label{eq:def-beta}
        \mathcal{F}\tw_j^\dir\equiv\beta_j^\dir w_j^\dir\quad\text{microlocally near }\gamma_j^\dir,
    \end{equation}
    for each $(j,\dir)\in\{1,2\}\times\{\inc,\out\}$.
    They admit the asymptotic formula
    \begin{align}\label{eq:beta-asympt}
        \beta_1^\dir=\left|\partial_{(x,\xi)} p_1(0,0)\right|^{-1/2}+\BigO{h},\qquad
        \beta_2^\dir=
        \left|c\partial_{(x,\xi)} p_2(0,0)\right|^{-1/2}+\BigO{h},
    \end{align}
    as $h\to+0$.
\end{lemma}

\begin{proof}[Proof of Proposition~\ref{prop:T-and-Tnf} when $c>0$ using Lemma~\ref{lem:beta-with-pos-c}]
    The definition \eqref{eq:def-beta} of $\beta_j^\dir$ shows the identity
    \begin{equation}\label{eq:T-matrix-ori-nf}
    T=\operatorname{diag}(\beta_1^\out\ \beta_2^\out)T_{\rm nf}\operatorname{diag}(\beta_1^\inc\ \beta_2^\inc)^{-1}.
\end{equation}
According to Proposition~\ref{Prop:AsymptoticsOfTransferMatrix},
one has
\begin{equation*}
    T_{\rm nf}=I+\begin{pmatrix}
        0&t_{12}\\t_{21}&0
    \end{pmatrix}+\BigO{h^{\frac2{m+1}}\logdelt}.
\end{equation*}
We obtain Proposition~\ref{prop:T-and-Tnf} by substituting this formula and \eqref{eq:beta-asympt} into \eqref{eq:T-matrix-ori-nf}.
\end{proof}

We will apply Proposition~\ref{Prop:FIO-WKB} to prove Lemma~\ref{lem:beta-with-pos-c}. To do this, we need the following lemmas.

\begin{lemma}
Assume that \eqref{eq:NoCaustic} holds. Then, the projection \eqref{eq:projdiffeo} is a diffeomorphism near $(0,0)$
\end{lemma}

\begin{proof}
It suffices to prove that the $1\times 1$ block $\frac{\pa x}{\pa y}(0,0)$ in the derivative $d\kappa(0,0)$ is invertible. If $\frac{\pa \xi}{\pa y}(0,0)=0$, then we have $\frac{\pa x}{\pa y}(0,0)\neq 0$ since $d\kappa(0,0)$ is invertible. Thus we may assume $\frac{\pa \xi}{\pa y}(0,0)\neq 0$.
Differentiating the equation $(\kappa^*p_1)(y,\eta)=\eta$ with respect to $y$, we have
\begin{align*}
0=\frac{\pa x}{\pa y}(0,0)\partial_{x}p_1(0,0)+\frac{\pa \xi}{\pa y}(0,0)\partial_{\xi}p_1(0,0).
\end{align*}
Since $\partial_{\xi}p_1(0,0)\neq 0$ by \eqref{eq:NoCaustic}, we conclude $\frac{\pa x}{\pa y}(0,0)\neq 0$.
\end{proof}

In the next lemma, we compute the index $\upsilon$ appearing in Proposition~\ref{Prop:FIO-WKB} in our situation.

\begin{lemma}\label{lem:FIO-index}
    Put $F(y)=\int_0^yf(t)dt$. 
    Under \eqref{eq:NoCaustic} and \eqref{eq:Not-changed}, one has 
    \begin{align}
        &\operatorname{sgn}\partial^2_{(y,\eta)}(-y\cdot\eta+\psi(0,\eta))|_{(y,\eta)=(0,0)}=0,\label{eq:FIO-index1}\\
        &\operatorname{sgn}\partial^2_{(y,\eta)}(F(y)-y\cdot\eta+\psi(0,\eta))|_{(y,\eta)=(0,0)}=0. \label{eq:FIO-index2}   
    \end{align}
\end{lemma}
\begin{proof}
    We see that the Hessian matrix is computed as
    \begin{equation}\label{eq:FIO-indexpf1}
        \partial^2_{(y,\eta)}(F(y)-y\cdot\eta+\psi(0,\eta))|_{(y,\eta)=(0,0)}=
        \begin{pmatrix}
            f'(0)&-1\\-1&\partial_\eta^2\psi(0,0)
        \end{pmatrix}.
    \end{equation}
    The determinant of this matrix is $-1$ if $F$ vanishes identically. 
    This 
    proves \eqref{eq:FIO-index1}.  
For \eqref{eq:FIO-index2}, we \blue{will} calculate the determinant of \eqref{eq:FIO-indexpf1} as
\begin{align}\label{eq:FIO-indexpfid}
\det  \left(\partial^2_{(y,\eta)}(F(y)-y\cdot\eta+\psi(0,\eta))\right)|_{(y,\eta)=(0,0)}= -\frac{\partial_\xi(\kappa^*p_1)(0,0)}{\partial_\xi p_1(0,0)}\frac{\partial_\xi p_2(0,0)}{\partial_\xi(\kappa^*p_2)(0,0)}.
\end{align}
We observe that \eqref{eq:FIO-index2} 
is equivalent to this determinant being negative.
Then, \eqref{eq:FIO-index2} follows from \eqref{eq:Not-changed}, \eqref{eq:FIO-indexpfid}, and the elementary identity $\sgn (\frac{a}{b}\cdot\frac{c}{d})=\sgn (\frac{b}{a}\cdot\frac{c}{d})$.

In the following, we prove the identity \eqref{eq:FIO-indexpfid}.
Let $g_j=g_j(x)$ be the function defined near $x=0$ by
\begin{equation*}
\kappa(g_1(x),0)=(x,\phi_1'(x)),\quad
\kappa(g_2(x),f(g_2(x)))=(x,\phi_2'(x)).
\end{equation*}
We observe that $g_1(0)=g_2(0)=0$ holds 
because
$\kappa(0,0)=(0,0)$ and \eqref{eq:p_jeikonal}.
We claim that
\begin{align}\label{eq:FIO-indexpf2}
\det\partial^2_{(y,\eta)}(F(y)-y\cdot\eta+\psi(0,\eta))=-\frac{g_1'(0)}{g_2'(0)}
\end{align}
holds.
The definition of the generating function, $
        \kappa(\partial_\eta\psi(x,\eta),\eta)=(x,\partial_x\psi(x,\eta))$,
    implies that one has $g_1(x)=\partial_\eta\psi(x,0)$ and $g_2(x)=\partial_\eta\psi(x,f(g_2(x))$.
    By differentiating in $x$, we obtain
    \begin{equation*}
        g_1'(x)=\partial_x\partial_\eta\psi(x,0),\quad
        g_2'(x)=\partial_x\partial_\eta\psi(x,f(g_2(x)))+g_2'(x)f'(g_2(x))\partial_\eta^2\psi(x,f(g_2(x)).
    \end{equation*}
    We in particular deduce that $g_1'(0)=\partial_x\partial_\eta \psi(0,0)$ from the first identity, and
    \begin{equation*}
f'(0)\partial_\eta^2\psi(0,0)-1=-\frac{g_1'(0)}{g_2'(0)}        
    \end{equation*}
where we use $g_1(0)=g_2(0)=0$ and $f(0)=0$. This with \eqref{eq:FIO-indexpf1} proves \eqref{eq:FIO-indexpf2}.
    
By the same calculation as \eqref{eq:gderp}, 
\begin{align}\label{eq:g1g2at0}
g_j'(0)=\frac{\partial_\xi(\kappa^*p_j)(0,0)}{\partial_\xi p_j(0,0)}\quad \text{for}\quad j=1,2.
\end{align}
Combining \eqref{eq:FIO-indexpf2} with \eqref{eq:g1g2at0}, we obtain \eqref{eq:FIO-indexpfid}.
\end{proof}

\begin{proof}[Proof of Lemma~\ref{lem:beta-with-pos-c}]
Let $\tilde{\gamma}_j^{\dir}$ ($(j,\dir)\in\{1,2\}\times\{\inc,\out\}$) be the four connected components of $(\{\xi=0\}\cap \{\xi=f(x)\})\setminus \{0\}$. Our assumption $c=(\partial_\xi (\kappa^*p_2)(0,0))^{-1}>0$ implies $\kappa(\tilde{\gamma}_j^{\dir})=\gamma_j^{\dir}$ in a neighborhood of $(0,0)$.
Since $\tw_j^{\dir}$ and $w_j^{\dir}$ are WKB solutions associated with $\tilde{\gamma}_j^{\dir}$ and $\gamma_j^{\dir}$ respectively, the uniqueness of the microlocal solutions on each $\gamma_j^\dir$ ($(j,\dir)\in
\{1,2\}\times\{\inc,\out\}$) (Proposition~\ref{Prop:matrixmicsol} $(i)$) implies that we can find constants $\beta_j^{\dir}$ satisfying \eqref{eq:def-beta}. Therefore, it suffices to prove the asymptotic formula \eqref{eq:beta-asympt} of $\beta_{j}^{\dir}$.

By Proposition~\ref{Prop:matrixmicsol} $(ii)$ and $(iii)$, we have
\begin{align}
&w_1^{\dir}(x)=e^{\frac{i}{h}\phi_1(x)}\begin{pmatrix}
c_1^{\dir}\\ 0
\end{pmatrix}+\BigO{h},\quad w_2^{\dir}(x)=e^{\frac{i}{h}\phi_2(x)}\begin{pmatrix}
0\\ c_2^{\dir}
\end{pmatrix}+\BigO{h}\quad \text{with} \label{eq:betacal1}\\
&c_j^{\dir}=\left|\partial_{(x,\xi)} p_j(0,0)\right|^{\frac{1}{2}}\left|\partial_\xi p_j(0,0)\right|^{-\frac{1}{2}}. \label{eq:betacal2}
\end{align}
Since $\tw_1^{\dir}={}^t(\mathbbm{1},0)+\BigO{h}$ and $\tw_2^{\dir}=e^{\frac{i}{h}F(\cdot)}\cdot{}^t(0,\mathbbm{1})+\BigO{h}$, Proposition \ref{Prop:FIO-WKB} with Lemma \ref{lem:FIO-index} implies
\begin{align}
\mathcal{F}\tw_1^\dir\equiv
\begin{pmatrix}
F_1(\mathbbm{1})\\
0
\end{pmatrix}+\BigO{h},\quad
\mathcal{F}\tw_2^\dir\equiv
\begin{pmatrix}
0\\
F_2(e^{\frac{i}{h}F(\cdot)})
\end{pmatrix}+\BigO{h}\label{eq:betacal3}
\end{align}
for $\dir=\inc,\out$. Since $F_1$ and $F_2$ are microlocally unitary Fourier integral operators in the sense of Definition \ref{def:FIO} $(iii)$, Proposition \ref{Prop:FIO-WKB} (with $\tilde{\phi}_1(x)=0$ and $\tilde{\phi}_2(x)=F(x)$) implies
\begin{align}
&F_1(\mathbbm{1})(x)=\sigma_1(x,h)e^{\frac{i}{h}\phi_1(x)},\quad F_2(e^{\frac{i}{h}F(\cdot)})=\sigma_2(x,h)e^{\frac{i}{h}\phi_2(x)}\quad \text{with}\label{eq:betacal4}\\
&\sigma_j(0,h)=\left|\frac{\partial_\xi(\kappa^*p_j)(0,0)}{\partial_\xi p_j(0,0)}\right|^{1/2}+\BigO{h}\quad (j=1,2),\label{eq:betacal5}
\end{align}
where the index $\upsilon$ appearing in Proposition \ref{Prop:FIO-WKB} is equal to $0$ by virtue of Lemma \ref{lem:FIO-index}. By \eqref{eq:def-beta}, \eqref{eq:betacal1}, \eqref{eq:betacal3}, and \eqref{eq:betacal4}, it must hold that
\begin{align*}
\beta_j^{\dir}c_j^{\dir}=\sigma_j(0,h)+\BigO{h^{\infty}}.
\end{align*}
Combining this with \eqref{eq:betacal2} and \eqref{eq:betacal5}, we obtain
\begin{align*}
\beta_j^{\dir}=|\partial_\xi(\kappa^*p_j)(0,0)|^{\frac{1}{2}} |\partial_{(x,\xi)} p_j(0,0)|^{-\frac{1}{2}}+\BigO{h}.
\end{align*}
Now we recall that $\partial_\xi(\kappa^*p_1)(0,0)=1$ and $c=(\partial_\xi(\kappa^*p_2)(0,0))^{-1}$, which shows \eqref{eq:beta-asympt}.

\end{proof}

\subsection{Proof of Proposition~\ref{prop:T-and-Tnf} when $c$ is negative}\label{SubS:negative-c}

Let us consider the case with negative $c$.
The major difference is the fact that the Hamiltonian flow induced by the Hamiltonian $\xi-f(x)$ has the inverse direction to that by $\kappa^*p_2(x,\xi)=a_0(x,\xi)(\xi-f(x))$ when $c=a_0(0,0)^{-1}$ is negative.

\begin{lemma}\label{lem:beta-with-neg-c}
    Assume $c<0$. 
    There exist constants $\beta_1^\inc$, $\til{\beta}_2^\inc$, $\beta_1^\out$, and $\til{\beta}_2^\out$ such that
    \begin{equation}\label{eq:def-beta2}
        \mathcal{F}\tw_1^\dir\equiv\beta_1^\dir w_1^\dir\quad \text{microlocally near }\gamma_1^\dir,\qquad(\dir=\inc,\out)
    \end{equation}
    and
    \begin{equation}\label{eq:def-beta3}
        \mathcal{F}\tw_2^\inc\equiv\til{\beta}_2^\out w_2^\out\quad \text{microlocally near }\gamma_2^\out,\qquad
        \mathcal{F}\tw_2^\out\equiv\til{\beta}_2^\inc w_2^\inc\quad \text{microlocally near }\gamma_2^\inc.
    \end{equation}
    They admit the asymptotic formula
    \begin{align}\label{eq:beta-asympt2}
        \beta_1^\dir=\left|\partial_{(x,\xi)} p_1(0,0)\right|^{-1/2}+\BigO{h},\qquad
        \til{\beta}_2^\dir=
        \left|c\partial_{(x,\xi)} p_2(0,0)\right|^{-1/2}+\BigO{h},
    \end{align}
    as $h\to+0$.
\end{lemma}

\begin{proof}
Let $\tilde{\gamma}_j^{\dir}$ ($(j,\dir)\in\{1,2\}\times\{\inc,\out\}$) be four connected components of $(\{\xi=0\}\cap \{\xi=f(x)\})\setminus \{0\}$. Our assumption $c=(\partial_\xi (\kappa^*p_2)(0,0))^{-1}<0$ implies 
\begin{align*}
\kappa(\tilde{\gamma}_1^{\dir})=\gamma_1^{\dir},\quad \kappa(\tilde{\gamma}_2^{\inc})=\gamma_2^{\out},\quad \kappa(\tilde{\gamma}_2^{\out})=\gamma_2^{\inc}
\end{align*}
in a neighborhood of $(0,0)$. This is the main difference from the case $c>0$. From these relations, we obtain \eqref{eq:def-beta2} and \eqref{eq:def-beta3} with constants $\beta_1^{\dir}$ and $\tilde{\beta}_2^{\dir}$. The asymptotic formula \eqref{eq:beta-asympt2} of $\beta_{j}^{\dir}$ and $\tilde{\beta}_2^{\dir}$ can be proved similarly to the proof of Lemma~\ref{lem:beta-with-pos-c}.
\end{proof}

\begin{proof}[Proof of Proposition~\ref{prop:T-and-Tnf} when $c<0$]
    Let $w$ satisfy the microlocal Cauchy problem \eqref{eq:msPCauchy} for the original operator $\scP$ with the initial condition given by $\alpha_1^\inc$, $\alpha_2^\inc$, and let $\alpha_1^\out$ and $\alpha_2^\out$ be the outgoing data \eqref{eq:outgoing-data} of $w$.
    By the definition \eqref{eq:def-beta} of $\beta_j^\dir$ and the transfer matrix, 
    we have shown
    the identity
    \begin{equation*}
\begin{pmatrix}\alpha_1^\out\\\alpha_2^\inc\end{pmatrix}
        =\operatorname{diag}(\beta_1^\out\ \til{\beta}_2^\inc)
        T_{\rm nf}\diag(\beta_1^\inc\ \til{\beta}_2^\out)^{-1}
        \begin{pmatrix}\alpha_1^\inc\\\alpha_2^\out\end{pmatrix},\qquad \begin{pmatrix}
            \alpha_1^\out\\\alpha_2^\out
        \end{pmatrix}
        =T
        \begin{pmatrix}
            \alpha_1^\inc\\\alpha_2^\inc
        \end{pmatrix}.
    \end{equation*}
    An algebraic computation\footnote{
    The identity
    \begin{equation*}
        \begin{pmatrix}
            x\\w
        \end{pmatrix}
        =
        \begin{pmatrix}
            a_{11}&a_{12}\\
            a_{21}&a_{22}
        \end{pmatrix}
        \begin{pmatrix}
            y\\z
        \end{pmatrix}
        \iff
        \begin{pmatrix}
            y\\z
        \end{pmatrix}
        =
        \begin{pmatrix}
            a^{11}&a^{12}\\
            a^{21}&a^{22}
        \end{pmatrix}
        \begin{pmatrix}
            x\\w
        \end{pmatrix}
    \end{equation*}
    with $(a^{jk})_{j,k}=(a_{jk})_{j,k}^{-1}\in\mathcal{M}_2(\Cc)$ implies $z=a^{21}x+a^{22}w=a^{21}(a_{11}y+a_{12}z)+a^{22}w$. This gives $z=(1-a_{12}a^{21})^{-1}(a^{21}a_{11}y+a^{22}w)$.
    By substituting this identity into $x=a_{11}y+a_{12}z$, $x$ is also written 
    as a
    linear combination of $y$ and $w$.
    Then we obtain
    \begin{equation*}
        \begin{pmatrix}
            x\\z
        \end{pmatrix}
        =\frac1{1-a_{12}a^{21}}
        \begin{pmatrix}
            a_{11}&a_{12}a^{22}\\
            a^{21}a_{11}&a^{22}
        \end{pmatrix}
        \begin{pmatrix}
            y\\w
        \end{pmatrix}.
    \end{equation*}
    } shows 
    \begin{equation}\label{eq:T-matrix-ori-nf2}
        T=\operatorname{diag}(\beta_1^\out\ \til{\beta}_2^\out)
        \frac1{1-t^{21}t_{12}}
        \begin{pmatrix}
            t_{11}&t_{12}t^{22}\\
            t^{21}t_{11}&t^{22}
        \end{pmatrix}
        \operatorname{diag}(\beta_1^\inc\ \til{\beta}_2^\inc)^{-1},
    \end{equation}
    where $t_{jk}$ are the $(j,k)$-entry of  $T_{\rm nf}$, and $t^{jk}$ are those of $T_{\rm nf}^{-1}$.
Now we observe that
\begin{align*}
\frac1{1-t^{21}t_{12}}
\begin{pmatrix}
t_{11}&t_{12}t^{22}\\
t^{21}t_{11}&t^{22}
\end{pmatrix}=I+\begin{pmatrix}
0&t_{12}\\t_{21}&0
\end{pmatrix}+\BigO{h^{\frac2{m+1}}\logdelt}
\end{align*}
hold since one has
\begin{equation*}
T_{\rm nf}=I+\begin{pmatrix}
0&t_{12}\\t_{21}&0
\end{pmatrix}+\BigO{h^{\frac2{m+1}}\logdelt},\,\, T_{\rm nf}^{-1}=I-\begin{pmatrix}
0&t_{12}\\t_{21}&0
\end{pmatrix}+\BigO{h^{\frac2{m+1}}\logdelt}
\end{equation*}
and $t_{12},t_{21}=\BigO{h^{\frac{1}{m+1}}}$ according to Proposition~\ref{Prop:AsymptoticsOfTransferMatrix} and the Neumann series expansion.
We obtain Proposition~\ref{prop:T-and-Tnf} by substituting these formulae and \eqref{eq:beta-asympt2} into \eqref{eq:T-matrix-ori-nf2}.
\end{proof}

\subsection{Computation of the asymptotics}\label{S:computation}
We provide an outline of the computation of the asymptotics of the transfer matrix given in Theorem~\ref{Th:LSM}.
It is deduced by combining the formulae written in Theorem~\ref{Thm:Matrix-NF}, Propositions~\ref{Prop:AsymptoticsOfTransferMatrix}, \ref{prop:T-and-Tnf},
and the following lemma.
\begin{lemma}\label{lem:prelim-for-compute}
    One has
    \begin{equation}\label{eq:sign-s-c}
        s=\operatorname{sgn}\,c,
    \end{equation}
    where $s$ and $c$ are defined in \eqref{eq:def-s-tau-mu} and Theorem~\ref{Thm:Matrix-NF} (iv), respectively.
    One also has
    \begin{equation}\label{eq:sym-Hp12}
        H_{p_1}^mp_2(0,0)=-c^{m-1}H_{p_2}^mp_1(0,0).
    \end{equation}
\end{lemma}

\begin{proof}
\underline{Proof of \eqref{eq:sign-s-c}:}
    According to the condition \eqref{eq:Not-changed}, the sign of $c$ coincides with that of the product $(\partial_\xi p_1(0,0))(\partial_\xi p_2(0,0))$ under the condition \eqref{eq:NoCaustic} (the product does not vanish).
    By definition of $\Theta(p_j)\in[-\pi/2,\pi/2)$, the vector $\tau(\Theta(p_j))$ is the unit vector parallel to the Hamiltonian vector $H_{p_j}(0,0)$.
    This implies that the sign of $(H_{p_j}(0,0),\tau(\Theta(p_j)))_{\Rr^2}$ coincides with that of $(\cos(\Theta(p_j)))\partial_\xi p_j(0,0)$ if the latter does not vanish.
    Note that the condition \eqref{eq:NoCaustic}: $\partial_\xi p_j(0,0)\neq0$  is equivalent to $\Theta(p_j)\neq-\pi/2$.
    This condition also implies $\cos(\Theta(p_j))>0$, and consequently \eqref{eq:sign-s-c}.

    Otherwise, we use the rotation map $\kappa_\theta$ as in Proposition~\ref{Prop:matrixmicsol}.
    Recall that the value 
    \begin{equation*}
        (H_{p_{j,\theta}}(0,0),\tau(\Theta(p_{j,\theta})))_{\Rr^2}=(H_{p_j}(0,0),\tau(\Theta(p_j)))_{\Rr^2},\quad(p_{j,\theta}:=\kappa_\theta^*p_j)
    \end{equation*}
    is invariant for small $\theta\ge0$.
    For small $\theta>0$, the equality $\Theta(p_{j,\theta})=\Theta(p_j)+\theta$ implies $\partial_\xi p_{j,\theta}(0,0)\neq0$.
    This reduces the problem to the case which is already treated above.    

    \noindent
    \underline{Proof of \eqref{eq:sym-Hp12}:}
    By employing $\kappa$ such that $\kappa^*p_1(x,\xi)=\xi$, one can easily check for each $k\ge1$ that 
    \begin{equation}\label{eq:Hp-many-times}
        \begin{aligned}
            &H_{p_1}^kp_2=\partial_x^k(\kappa^*p_2),\\
            &H_{p_2}^kp_1=-H_{\kappa^*p_2}^{k-1}\partial_x(\kappa^*p_2)
            =-[(\partial_\xi(\kappa^*p_2))\partial_x]^{k-1}\partial_x(\kappa^*p_2).
        \end{aligned}
    \end{equation}
    By the definition of the contact order, $H_{p_1}^k p_2(0,0)=\partial_x^k(\kappa^*p_2)(0,0)=0$ holds for $1\le k\le m-1$.
    Plugging this into the second equality of \eqref{eq:Hp-many-times}, one obtains
    \begin{equation*}
        H_{p_2}^mp_1(0,0)=-c^{-(m-1)}\partial_x^m(\kappa^*p_2)(0,0)=-c^{-(m-1)}H_{p_1}^mp_2(0,0).
    \end{equation*}
\end{proof}

\begin{proof}[Proof of Theorem~\ref{Th:LSM}]

Due to Proposition \ref{prop:scalarmicsol} $(iv)$ and the definition, the transfer matrix $T$ is invariant under the conjugation of the metaplectic operator $M_{\theta}$ for sufficiently small $\theta>0$. Moreover, we can find sufficiently small $\theta>0$ such that
the symbol $p_{j,\theta}:=\kappa_{\theta}^*p_j$ of $M_{\theta}^{-1}P_jM_{\theta}$ satisfies \eqref{eq:NoCaustic} by Condition \ref{C:RealPrincipal}, where $\kappa_{\theta}$ is the rotation defined in \eqref{eq:rotation}.
Thus, we may assume \eqref{eq:NoCaustic}.

As in Theorem~\ref{Thm:Matrix-NF} (iv), $r_1(0,0)$ and $r_2(0,0)$ coincide with $r_1(0,0)$ and $cr_2(0,0)$ modulo $\BigO{h}$.
Then by combining Propositions~\ref{Prop:AsymptoticsOfTransferMatrix} and \ref{prop:T-and-Tnf}, the anti-diagonal entries of the transfer matrix $T$ are given by $q_1(0,0)\omega_1$ and $q_2(0,0)\omega_2$ modulo $\BigO{h}$, where $\omega_1$ and $\omega_2$ are defined by
\begin{equation}\label{eq:omega12}
    \omega_1=\left|\frac{c}{c'}\right|^{\frac12}\omega,\quad\omega_2=|cc'|^{\frac12}\overline{\omega}.
\end{equation}
Here, $\omega$ is defined by \eqref{Eq:omega} in terms of the $m$-th derivative at $x=0$ of the function $f(x)$ appearing in the normal form.
According to the formula \eqref{Eq:Derivative-f} in Theorem~\ref{Thm:Matrix-NF}, this function $f$ satisfies $f^{(m)}(0)=-cH^m_{p_1}p_2(0,0)$.
Then $\omega$ is rewritten as
\begin{equation*}
    \omega=2|c|^{-\frac1{m+1}}\mu_m\left(\frac{-\operatorname{sgn}(sH_{p_1}^mp_2(0,0))\pi}{2(m+1)}\right)\bm{\Gamma}\left(\frac{m+2}{m+1}\right)\left(\frac{(m+1)!}{|H_{p_1}^mp_2(0,0)|}\right)^{\frac1{m+1}}.
\end{equation*}
Here, we also applied \eqref{eq:sign-s-c}.
The asymptotic formula for $\omega_1$ given in Theorem~\ref{Th:LSM} follows from this with \eqref{eq:omega12} and the definition of $c$ and $c'$. 
More concretely, one obtains the identity $|c/c'|^{1/2}|c|^{-1/(m+1)}=|c'|^{-1/(m+1)}$ with $c'=|\partial_{(x,\xi)}p_1(0,0)|/|\partial_{(x,\xi)}p_2(0,0)|$ which is obvious for $m=1$, and is implied by $|c|=c'$ in the tangential case $m\ge2$ (see Theorem~\ref{Thm:Matrix-NF} (v)).

For the formula of $\omega_2$, we also use \eqref{eq:sym-Hp12}.
The difference between $\omega_1$ and $\omega_2$ caused by the complex conjugate only appears in $\mu_m$ and $\overline{\mu_m}$ since the other factors are real.
Moreover, $\mu_m$ is real when $m$ is even.
When $m$ is odd, one has $\overline{\mu_m(\theta)}=\mu_m(-\theta)$ on one hand, and 
\begin{equation*}
    H_{p_1}^mp_2(0,0)=-c^{m-1}H_{p_2}p_1(0,0)=-|c|^{m-1}H_{p_2}p_1(0,0),
\end{equation*}
on the other hand.
We conclude here that 
\begin{equation*}
    \overline{\mu_m\left(\frac{-\operatorname{sgn}(sH_{p_1}^mp_2(0,0))\pi}{2(m+1)}\right)}=\mu_m\left(\frac{-\operatorname{sgn}(sH_{p_2}^mp_1(0,0))\pi}{2(m+1)}\right).
\end{equation*}

Concerning the absolute value, \eqref{eq:omega12} implies 
\begin{equation*}
    |\omega_2|=|c'||\omega_1|.
\end{equation*}
The equality \eqref{eq:sym-Hp12} shows 
\begin{equation*}
    |c'|\left(\frac{|\partial_{(x,\xi)}p_2(0,0)|}{|\partial_{(x,\xi)}p_1(0,0)|}\frac{(m+1)!}{|H_{p_1}^mp_2(0,0)|}\right)^{\frac1{m+1}}
    =\left(\frac{|\partial_{(x,\xi)}p_1(0,0)|}{|\partial_{(x,\xi)}p_2(0,0)|}\frac{(m+1)!}{|H_{p_2}^mp_1(0,0)|}\right)^{\frac1{m+1}}.
\end{equation*}
This is the only factor of $\omega_2$ different from $\omega_1$ other than $\mu_m$.
\end{proof}
        	\appendix

\section{The stationary phase method for a degenerate phase}

\begin{lemma}\label{Lemma:degst}
Let $I \subset \Rr$ be a bounded, $h$-independent interval. Let $F \in C^\infty(I,\Rr)$ be a function whose derivative vanishes nowhere but at $x=0$, at a finite order $m \geq 1$. Let $a \in C^\infty(I,\Cc)$. 
Then
\begin{equation}\label{Eq:degstEstimate}
\int_I a(y)e^{\frac{i}{h}F(y)}dy \leq C \left[ h^{\frac{1}{m+1}} \norm{a}{L^\infty(I)} + h^{\frac{2}{m+1}}\log(1/h)^{\delta_{m,1}}\norm{a'}{L^\infty(I)} \right]
\end{equation}
where $\delta_{m,1}$ is the Kronecker delta. 
Moreover, if $J \Subset I \cap (0,+\infty)$ is $h$-independent, then for all $I \ni x_0 < 0$ and $x \in J$,
\begin{multline}\label{Eq:degstExpansion}
	\int_{x_0}^x a(y)e^{\frac{i}{h}(F(y)-F(0))}dy \\
	=2\mu_m\left(\frac{\sgn(F^{(m+1)}(0))\pi}{2(m+1)}\right)
	\pmb\Gamma\left(\frac{m+2}{m+1}\right) 
	\left(\frac{m+1}{|F^{(m+1)}(0)|}\right)^{\frac{1}{m+1}}
	a(0)h^{\frac{1}{m+1}}
	+ \BigOO{h^{\frac{2}{m+1}}}{L^\infty(J)}
\end{multline}
where $\pmb\Gamma$ is Euler's Gamma function and $\mu_m$ is defined by 
$\mu_m(\theta) := 2^{-1}(e^{i\theta} + e^{i(-1)^{m+1}\theta})$ for $\theta$ in $\Rr$.
\end{lemma}

\begin{proof}
The estimate \eqref{Eq:degstEstimate} comes from an integration by parts. 
For the expansion \eqref{Eq:degstExpansion}, perform an $h$-independent change of variable on $\int e^{\frac{i}{h}y^{m+1}}dy$. The last integral is computed with the residue theorem (Fresnel-type integral). We send the reader to \cite[Chap. VIII]{Ste} for more details.
\end{proof}

\section{Construction of WKB solutions for the reduced equation}\label{App:WKB}

We consider the reduced operator of Section \ref{S:CFReduced} defined by
\begin{equation*}
    P=P_{\rm nf}=\begin{pmatrix}
        hD_x&hR_1\\hR_2&hD_x-f(x)
    \end{pmatrix}.
\end{equation*}
Let $\Omega$ be a neighborhood of $(0,0)$. Let $\rho_j^\dir \in \gamma_j^\dir$ for $\dir = \inc,\out$ and $j = 1,2$. We define WKB solutions near each of these four points. Here, $\gamma_j^\inc$ denotes $\Gamma_j \cap \{\xi < 0\}$ and $\gamma_j^\out$ denotes $\Gamma_j \cap \{\xi > 0\}$ as illustrated on Figure \ref{Fig:Crossing}.

\begin{figure}[h]
    \begin{minipage}{\linewidth}
    {\centering
    \includegraphics[width=\linewidth,page=2]{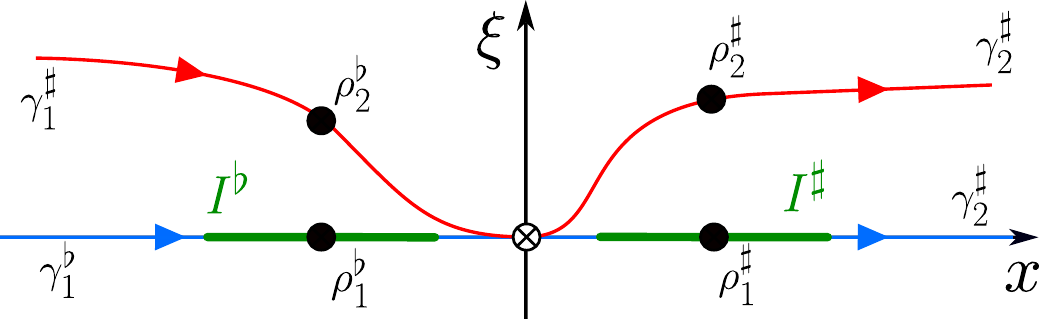}
    \caption{Crossing of Hamiltonian flows for the reduced case}\label{Fig:Crossing}
    }
    \end{minipage}
\end{figure}

\begin{prop}\label{prop:modelWKBconst}
Let $I^\inc \Subset (-\infty,0)$ and $I^\out \Subset (0,+\infty)$ be intervals such that $\rho_j^\dir \in I^\dir \times \Rr_\xi$. For each $j = 1,2$ and $\dir = \inc,\out$ there exist a microlocal WKB solution $\tw_j^\dir$ of $Pu=0$ near $\rho_j^\dir$ satisfying
\begin{equation}\label{Eq:WKBOrderh}
    \tw_1^\dir(x) =
    \begin{pmatrix}
    1 + \BigOO{h}{C^\infty_b}\\ \BigOO{h}{C^\infty_b}
    \end{pmatrix}
    \stext{and}
    \tw_2^\dir(x) = e^{\frac{i}{h}F(x)}
    \begin{pmatrix}
    \BigOO{h}{C^\infty_b}\\ 1 + \BigOO{h}{C^\infty_b}
    \end{pmatrix} \quad \text{for } x \in I^\dir,
\end{equation}
where $F(x)=\int_0^x f(y)dy$. Moreover, 
\begin{align*}
\mathrm{WF}_h(\tw_1^\dir(x)) \subset \{(x,0)\in T^*\Rr\mid x\in \Rr\},\quad \mathrm{WF}_h(\tw_2^\dir(x)) \subset \{(x,f(x)\mid x\in \Rr)\}
\end{align*}
by Example \ref{ex:WF}.

\end{prop}

\begin{proof}
The proof is more or less standard and is known as the WKB construction.
The construction of the four solutions ($\dir = \inc,\out$ and $j = 1,2$) is very similar, so we only construct $\tw_2^\inc$. Without loss of generality,
we may assume that the integral kernel $k_j$ of $R_j$ satisfies $\supp(k_j) \Subset I^\inc \times I^\inc$. We look for $\tw_2^\inc$ under the WKB form $\tw_2^\inc(x) \sim e^{\frac{i}{h}\phi(x)}\sum_{k=0}^\infty \tw_k(x)h^k$ where the amplitude $\tw_k(x) = (\tw_{k,1}(x),\tw_{k,2}(x))$ is vector-valued. 
We have, for a fixed $N \in \Nn$,
\begin{align}\label{Eq:Pw(x)}
    e^{-\frac{i}{h}\phi(x)}P e^{\frac{i}{h}\phi(x)}\sum_{k=0}^Nh^k \tw_k(x) &= 
    \begin{pmatrix}
        \phi'(x)\tw_{0,1}(x)\\
        (\phi'(x) - f(x))\tw_{0,2}(x)
    \end{pmatrix}\nonumber\\
    &+ \sum_{k=1}^{N} h^k 
    \begin{pmatrix}
         \phi'(x)\tw_{k,1}(x) -i\tw_{k-1,1}'(x)
        + (R_1)_{f}  \tw_{k-1,2}(x)   \\
        (\phi'(x) - f(x))\tw_{k,2} -i\tw_{k-1,2}'(x)(x) 
        + (R_2)_{f} \tw_{k-1,1}(x) 
    \end{pmatrix}\nonumber\\
    &+ h^{N+1}\begin{pmatrix}
         -i\tw_{N,1}'(x)
        + (R_1)_{f}  \tw_{N,2}(x)   \\
        -i\tw_{N,2}'(x)(x) 
        + (R_2)_{f} \tw_{N,1}(x)
    \end{pmatrix}.
\end{align}
Here, $(R_j)_{f}$ is the pseudo-differential operator obtained from the conjugation of $R_j$ by $e^{\frac{i}{h}\phi(x)}$ and given by Lemma \ref{Prop:conjugatePDO}. Remark that the integral kernel of $(R_j)_{f}$ is $e^{\frac{i}{h}(\phi(y) - \phi(x))}k_j(x,y)$, which is also compactly supported in $I^\inc \times I^\inc$. For $\tw_2^\inc$ to be a microlocal solution, we require the terms for $0 \leq k \leq N$ of the above to vanish. We may assume $\tw_0\neq0$. A possible choice of $\phi$ and $\tw_0$ for the first term to vanish is
\begin{equation}\label{Eq:phiAndw0}
    \phi(x) := \int_0^x f(y)dy \stext{and} (\tw_{0,1}(x),\tw_{0,2}(x)) := (0,1) \text{ for all } x \in \Rr.
\end{equation}
The equation $\phi'-f=0$ satisfied by $\phi$ is called the \textit{eikonal equation} associated with $\Gamma_2 $\footnote{Considering the eikonal equation $\phi'=0$ associated with $\Gamma_1$ would yield a construction of $\tw_1^\inc$.}. The system of \textit{transport equations} is obtained by requiring the terms for $1 \leq k \leq N$ in the sum of \eqref{Eq:Pw(x)} to vanish. It is solved by
amplitudes $\tw_k \in C^\infty(I^\inc)$ defined inductively by
\begin{equation}\label{Eq:AmplitudesWKB2inc}
    \text{For all } x \in I^\flat, \; \begin{cases}
        \tw_{k+1,1}(x) &:= f(x)^{-1}\left(i\tw_{k,1}'(x) - (R_1)_{f} \tw_{k,2}(x)  \right)\vspace{0.2cm}\\ 
        \tw_{k+1,2}(x) &:= -i  \int_{x_0}^x (R_2)_{f} \tw_{k+1,1}(y)dy
    \end{cases},
\end{equation}
where $x_0 \in I^\inc$ is any point in $I^\inc$. 
We now define via a Borel summation (a slightly modified version of \cite[Theorem 4.15]{Zwo}) a smooth function $\tw_2^\inc \sim e^{\frac{i}{h}\phi(x)}\sum_{k=0}^\infty \tw_k(x)h^k$. Finally, using the support properties of $\tw_k$ combined with \cite[Theorem 8.13]{Zwo}, we show that $\tw_2^\inc$ is indeed a microlocal solution of $Pu = 0$ near $\rho_2^\inc$.
\end{proof}

    \end{spacing}
\end{document}